\pdfoutput=1 
\DeclareSymbolFont{AMSb}{U}{msb}{m}{n}
\documentclass[11pt,noamsfonts,a4paper]{amsart}
\usepackage[left=1in, right=1in, top=1in, bottom=1in]{geometry}

\usepackage{dynkin-diagrams}

\usepackage{mathtools}
\usepackage{braket}
\usepackage[charter]{mathdesign}
\usepackage[tracking]{microtype}
\usepackage{tabularx}
\usepackage{etoolbox}
\usepackage{subcaption}
\usepackage{graphicx}
\usepackage{caption}

\usepackage{comment}

\usepackage{amsmath,amsthm}
\newtheoremstyle{pineapple}%
  {1em}{1em}%
  {\itshape}{}%
  {\bfseries}{. ---}
  {0.5em}{}

\newtheoremstyle{durian}%
  {1em}{1em}%
  {}{}%
  {\bfseries}{. ---}
  {0.5em}{}

\makeatletter
\def\swappedhead#1#2#3{%
  \thmnumber{\@upn{\the\thm@headfont#2\@ifnotempty{#1}{.~}}}%
  \thmname{#1}%
  \thmnote{ {\the\thm@notefont(#3)}}}
\makeatother

\makeatletter
\newcommand*\rel@kern[1]{\kern#1\dimexpr\macc@kerna}
\newcommand*\widebar[1]{%
  \begingroup
  \def\mathaccent##1##2{%
    \rel@kern{0.8}%
    \overline{\rel@kern{-0.8}\macc@nucleus\rel@kern{0.2}}%
    \rel@kern{-0.2}%
  }%
  \macc@depth\@ne
  \let\math@bgroup\@empty \let\math@egroup\macc@set@skewchar
  \mathsurround\z@ \frozen@everymath{\mathgroup\macc@group\relax}%
  \macc@set@skewchar\relax
  \let\mathaccentV\macc@nested@a
  \macc@nested@a\relax111{#1}%
  \endgroup
}
\makeatother

\makeatletter
\def\@sect#1#2#3#4#5#6[#7]#8{%
  \edef\@toclevel{\ifnum#2=\@m 0\else\number#2\fi}%
  \ifnum #2>\c@secnumdepth \let\@secnumber\@empty
  \else \@xp\let\@xp\@secnumber\csname the#1\endcsname\fi
  \@tempskipa #5\relax
  \ifnum #2>\c@secnumdepth
    \let\@svsec\@empty
  \else
    \refstepcounter{#1}%
    \edef\@secnumpunct{%
      \ifdim\@tempskipa>\z@ 
        \@ifnotempty{#8}{.~}%
      \else
        \@ifempty{#8}{.}{.~}%
      \fi
    }%
    \@ifempty{#8}{%
      \ifnum #2=\tw@ \def\@secnumfont{\bfseries}\fi}{}%
    \protected@edef\@svsec{%
      \ifnum#2<\@m
        \@ifundefined{#1name}{}{%
          \ignorespaces\csname #1name\endcsname\space
        }%
      \fi
      \@seccntformat{#1}%
    }%
  \fi
  \ifdim \@tempskipa>\z@ 
    \begingroup #6\relax
    \@hangfrom{\hskip #3\relax\@svsec}{\interlinepenalty\@M #8\par}%
    \endgroup
    \ifnum#2>\@m \else \@tocwrite{#1}{#8}\fi
  \else
  \def\@svsechd{#6\hskip #3\@svsec
    \@ifnotempty{#8}{\ignorespaces#8\unskip
       \@addpunct.}%
    \ifnum#2>\@m \else \@tocwrite{#1}{#8}\fi
  }%
  \fi
  \global\@nobreaktrue
  \@xsect{#5}}
\makeatother

\makeatletter
\def\@seccntformat#1{%
  \protect\textup{\protect\@secnumfont
    \ifnum\pdfstrcmp{subsection}{#1}=0 \bfseries\fi
    \csname the#1\endcsname
    \protect\@secnumpunct
  }%
}
\makeatother

\DeclareMathOperator{\Div}{Div}
\DeclareMathOperator{\Sym}{Sym}
\DeclareMathOperator{\coker}{coker}
\DeclareMathOperator{\gr}{gr}
\DeclareMathOperator{\res}{res}
\DeclareMathOperator{\Pic}{Pic}
\DeclareMathOperator{\rolldet}{roll\, det}

\DeclareMathOperator{\rollpf}{roll\,pf}
\DeclareMathOperator{\pf}{pf}

\DeclareMathOperator{\Proj}{Proj}

\theoremstyle{pineapple}
\newtheorem{IntroTheorem}{Theorem}
\newtheorem*{IntroTheorem*}{Theorem}

\swapnumbers
\newtheorem{Theorem}[subsection]{Theorem}
\newtheorem{Lemma}[subsection]{Lemma}
\newtheorem{Proposition}[subsection]{Proposition}
\newtheorem{Corollary}[subsection]{Corollary}
\newtheorem*{Corollary*}{Corollary}

\theoremstyle{durian}

\usepackage[dvipsnames]{xcolor}
\usepackage{hyperref} 
\hypersetup{
  colorlinks = true,
  linkcolor = Fuchsia,
  urlcolor = Fuchsia,
  citecolor = ForestGreen,
  linkbordercolor = {white}}
\usepackage{tikz-cd}
\usepackage{tikz}
\usetikzlibrary{arrows, matrix, fit}

\tikzset{
  symbol/.style={
    draw=none,
    every to/.append style={
      edge node={node [sloped, allow upside down, auto=false]{$#1$}}}
  }
}

\usepackage{enumitem}
\setlist[1]{labelindent=\parindent}
\setlist[1]{labelsep=0.5em}
\setlist[enumerate,1]{label={\upshape (\roman*)}, ref={\upshape (\roman*)}}

\makeatletter
\newcommand{\leqnomode}{\tagsleft@true\let\veqno\@@leqno}
\newcommand{\reqnomode}{\tagsleft@false\let\veqno\@@eqno}
\makeatother

\usepackage[utf8]{inputenc}

\tikzset{>={Straight Barb[length=2pt,width=4pt]}, commutative diagrams/arrow style=tikz}

\makeatletter
\let\c@equation\c@subsection

\makeatother

\makeatletter
\newcommand{\preceqdot}{\mathrel{\mathpalette\pr@ceqd@t\relax}}
\newcommand{\pr@ceqd@t}[2]{%
  \begingroup
  \sbox\z@{$#1\prec$}\sbox\tw@{$#1\preceq$}%
  \dimen@=\dimexpr\ht\tw@-\ht\z@\relax
  {\preceq}%
  \mkern-5mu
  \raisebox{\dimen@}{$\m@th#1\cdot$}%
  \endgroup
}
\makeatother

\makeatletter
\newcommand*{\coloneqq}{\mathrel{\rlap{%
           \raisebox{0.3ex}{$\m@th\cdot$}}%
           \raisebox{-0.3ex}{$\m@th\cdot$}}%
           =}
\newcommand{\eqqcolon}{=%
           \mathrel{\rlap{%
           \raisebox{0.3ex}{$\m@th\cdot$}}%
           \raisebox{-0.3ex}{$\m@th\cdot$}}}
\makeatother

\newcommand{\parref}[1]{{\bf\ref{#1}}}

\newcommand{\kk}{\mathbf{k}}

\newcommand{\PP}{\mathbf{P}}

\newcommand{\Gr}{\mathbf{Gr}}
\newcommand{\SGr}{\mathbf{SGr}}
\newcommand{\OGr}{\mathbf{OGr}}
\newcommand{\sO}{\mathcal{O}}

\makeatletter
\newcommand{\smallbullet}{} 
\DeclareRobustCommand\smallbullet{%
  \mathord{\mathpalette\smallbullet@{0.75}}%
}
\newcommand{\smallbullet@}[2]{%
  \vcenter{\hbox{\scalebox{#2}{$\m@th#1\bullet$}}}%
}
\makeatother

\newcommand{\subsectiondash}[1]{\subsection{#1}\textbf{---}\;}
\newcommand{\citeSP}[1]{\cite[\href{https://stacks.math.columbia.edu/tag/#1}{#1}]{stacks-project}}

\title{Frobenius--Tschirnhausen ampleness}
\author{Raymond Cheng}
\author{Emre Alp {{\"O}zavc{\i}}}
\address{SB MATH CAG \\
  EPFL \\
  Station 8 \\
  1015 Lausanne \\
  Switzerland
}
\email{raymond.cheng@epfl.ch}
\email{emre.ozavci@epfl.ch}

\begin{document}
\begin{abstract}
We study smooth projective varieties whose Frobenius-trace kernel, also known
as the Tschirnhausen bundle associated with the Frobenius morphism, is ample.
We show that any non-constant morphism out of such a variety must be finite onto
its image, providing strong evidence that these varieties must be Fano
varieties of Picard rank \(1\). Using infinitesimal representation theory and
by constructing special Frobenius splittings, we show that generalized
Grassmannian of classical type and type \(\mathrm{G}_2\) have ample
Frobenius-trace kernel, with the exception of certain low characteristic examples
which are explained by the existence of exotic isogenies of the associated
algebraic groups.
\end{abstract}
\maketitle
\setcounter{tocdepth}{1}
\thispagestyle{empty}
\section*{Introduction}

Positivity of natural vector bundles often imposes strong restrictions on the
global geometry of a smooth projective variety \(X\), the prototypical
statement being Mori's theorem from \cite{Mori:Hartshorne} on ampleness of
the tangent bundle: \(X\) must be projective space \(\PP^n\).
Relaxing ampleness to nefness, Campana and Peternell
conjectured in \cite{CP:Nef} that \(X\), at least over the complex numbers and
up to passing to an \'etale cover, must fibre into rational homogeneous spaces
over its Albanese variety. Another series of statements of a similar flavour
pertain to \(X\) which are \emph{Tschirnhausen ample}, by which we mean
that for any finite surjective morphism \(f \colon Y \to X\) from another
smooth projective variety, the associated Tschirnhausen bundle
\(\mathcal{E}_f \coloneqq (f_*\sO_Y/\sO_X)^\vee\) is ample.  Lazarsfeld
observed in \cite{Lazarsfeld:Barth} that \(\PP^n\) has this property. Peternell
and Sommese later showed in \cite{PS:Coverings-I} that any such \(X\) must be
simply connected and has no contractions; in particular, complex Fano
examples must have Picard rank \(1\). Kim and Manivel
found in \cite{Kim:Barth, Kim:Quadrics, Manivel:AVB, KM:Coverings} that, over
the complex numbers, quadrics of dimension at most \(6\), Grassmannians, and
Lagrangian Grassmannians are Tschirnhausen ample; they conjectured the same
for any complex rational homogeneous space of Picard rank \(1\).

Call a smooth projective variety \(X\) over a field of positive characteristic
\emph{Frobenius--Tschirnhausen ample} if the Tschirnhausen bundle
\(\mathcal{E}_X \coloneqq (F_*\sO_X/\sO_X)^\vee\) associated with its absolute
Frobenius morphism \(F \colon X \to X\) is ample. As explained below, this
property is best understood as a common specialization of the conditions
described above. Few general restrictions and examples are known: In their
initial work \cite{CRP}, Carvajal-Rojas and Patakfalvi showed that such \(X\)
must be a Fano variety which neither has a generically smooth fibration nor a
divisorial contraction onto another smooth variety; moreover, they found that
\(\PP^n\) and quadrics of dimension at least \(3\)---with the exception of the
quadric threefold in characteristic \(p = 2\)!---are Frobenius--Tschirnhausen
ample. Carvajal-Rojas and the second named author studied this property for
toric varieties in \cite{CRO:FrobeniusToric}, showing that the only
example amongst toric varieties is \(\PP^n\). Finally, Mallory found
a  geometric obstruction in \cite{Mallory} which implies, for instance, that Fano
varieties containing projective spaces of low anti-canonical degree are not
Frobenius--Tschirnhausen ample.

The aim of this work is to provide new restrictions on and examples of
Frobenius--Tschirnhausen ample varieties. Our first result completes the
initial results of Carvajal-Rojas and Patakfalvi and shows that such varieties
have no nontrivial contractions:

\begin{IntroTheorem}\label{intro-no-contractions}
Let \(X\) be a smooth projective variety which is Frobenius--Tschirnhausen ample.
Then any non-constant morphism \(f \colon X \to Y\) is finite onto its image.
\end{IntroTheorem}

This follows from the more technical result \parref{formal-G1} that any
connected closed subscheme \(Z \subseteq X\) satisfies Hironaka--Matsumara's G1
property: the formal functions on the completion of \(X\) along \(Z\) are just
the constants. Consequently, every semiample divisor on a
Frobenius--Tschirnhausen ample variety \(X\) is ample, and \(X\) has no
nontrivial extremal contractions. Since \(X\) is also a Fano variety, this
suggests that \(X\) ought to have Picard rank \(1\); we summarize this with the
following:

\begin{Corollary*}\label{intro-picard-rank}
If the contraction theorem holds for \(X\), then \(X\) has Picard rank \(1\).
\end{Corollary*}

Existence of contractions as in \cite[Theorem 3.7]{KM:Birational} is wide
open in positive characteristic, and is known only up to dimension \(3\)
by the work of Koll\'ar in \cite{Kollar:Contractions}. Contractions are also
known to exist for special classes of varieties such as rational homogeneous spaces,
or potentially more generally, those with nef tangent
bundle by \cite[Theorem 1.5]{KW:Nef}, and toric varieties; applying the
corollary to the latter shows that the only Frobenius--Tschirnhausen
ample toric variety is \(\PP^n\), giving another proof of the smooth case of
\cite[Theorem A, (a)\(\Leftrightarrow\)(b)]{CRO:FrobeniusToric}.

Which smooth Fano varieties of Picard rank \(1\) are Frobenius--Tschirnhausen
ample? Only \(\PP^n\) has this property in dimensions \(1\) and \(2\).  Mallory
shows in \cite[Theorem 1.2]{Mallory} that only smooth quadrics away from
characteristic \(p = 2\) and \(\PP^3\) have this property in dimension \(3\);
in higher dimensions, he shows that smooth complete intersections in \(\PP^n\)
of total degree \(d = n\) or \(d = n-1\) are \emph{not}
Frobenius--Tschirnhausen ample. The same is expected for any \(2 < d \leq n\),
and the following provides some evidence toward this:

\begin{IntroTheorem}\label{intro-nonample-hypersurfaces}
Let \(p < d \leq n\). There exists smooth Fano hypersurfaces of degree \(d\) in
\(\PP^n\) over a field of characteristic \(p\) which are not
Frobenius--Tschirnhausen ample.
\end{IntroTheorem}

This follows from \parref{formal-special-hypersurfaces} and
\parref{formal-special-hypersurfaces-exist}, where an explicit family of
hypersurfaces is shown to be not Frobenius--Tschirnhausen ample. The defining
equations of these hypersurfaces are taken to be of a special form: In
particular, they have a nontrivial \emph{profile} with respect to \(p\) in the
sense of \cite{profiles}; the simplest examples are the \emph{\(q\)-bic} or
\emph{extremal} hypersurfaces of \cite{fano-schemes, KKPSSVW}, such as the
Fermat hypersurface of degree \(q + 1\) with \(q = p^e\) for some \(e \geq 1\).

Besides complete intersections, the most accessible smooth Fano varieties
of Picard rank \(1\) are projective rational homogeneous spaces
\(X = \mathbf{G}/\mathbf{P}\) where \(\mathbf{P}\) is a reduced maximal
parabolic subgroup of a reductive algebraic group \(\mathbf{G}\). A first
observation---one which already implicitly appeared in Mehta and Ramanathan's
seminal work on Frobenius splittings in \cite[Remark (ii) on
p.40]{MR:F-Splitting}---is that \(\mathcal{E}_X\) is always globally generated,
even for \(X\) not necessarily of Picard rank \(1\): see
\parref{homogeneous-globally-generated}. Representation-theoretic
methods relate ampleness of \(\mathcal{E}_X\) to the existence of special
Frobenius splittings on \(X\) and, away from certain low characteristic
exceptions, such splittings may be explicitly constructed for Grassmannian of
classical type and of type \(\mathrm{G}_2\); note that examples such as the
quadric threefold in characteristic \(p = 2\) show that conditions on
the characteristic are necessary. Our technique is notably different from prior works
in that we need not explicitly decompose \(F_*\sO_X\). The result is:

\begin{IntroTheorem}\label{intro-homogeneous-spaces}
A projective rational homogeneous space \(X\) of Picard rank \(1\) of classical
type or type \(\mathrm{G}_2\) is Frobenius--Tschirnhausen ample if and only if
\(X\) is not the
\begin{enumerate}
\item\label{intro-homogeneous-spaces.symplectic}
symplectic Grassmannian \(\mathbf{SGr}(k,2m)\) with \(2 \leq k \leq m\) in
characteristic \(p = 2\); or the
\item\label{intro-homogeneous-spaces.orthogonal}
orthogonal Grassmannian \(\mathbf{OGr}(m-1,2m+1)\) in characteristic
\(p = 2\); or the
\item\label{intro-homogeneous-spaces.G2}
\(\mathrm{G}_2\)-Grassmannian in characteristics \(p = 2\) or \(p = 3\).
\end{enumerate}
\end{IntroTheorem}

Ampleness follows from the stronger result \parref{homogeneous-theorem}, which
shows that the twist \(\mathcal{E}_X \otimes \sO_X(-1)\) by the
negative generator of \(\Pic X\) is globally generated. In particular, quadrics
of dimension \(\geq 4\) in characteristic \(p = 2\) are Frobenius--Tschirnhausen
ample, correcting a slight inaccuracy in \cite[Corollary 4.8]{CRP}.
Non-ampleness for the exceptions is established in
\S\parref{section-exceptions} and, in most cases, is explained by the existence
of exotic isogenies between the relevant algebraic groups. Optimistically, we
would guess that the generalized Grassmannian in the remaining exceptional
types are also Frobenius--Tschirnhausen ample, at least away from
characteristic \(p = 2\); more conservatively, we conjecture
that this is the case as soon as the prime is good, say \(p > 5\).

The quadric threefold \(Q\) in characteristic \(p = 2\) appears twice in this
list of exceptions, and non-ampleness of \(\mathcal{E}_Q\) is explained by the
existence of the finite
purely inseparable morphism
\begin{align*}
\psi \colon \PP^3 & \to Q = \{(x_0:\cdots:x_4) \in \PP^4 : x_0x_1 + x_2x_3 + x_4^2 = 0\} \\
(x_0:x_1:x_2:x_3) & \mapsto (x_0^2:x_1^2:x_2^2:x_3^2:x_0x_1+x_2x_3)
\end{align*}
albeit for different reasons: Viewing \(Q = \mathbf{SGr}(2,4)\) as in
\ref{intro-homogeneous-spaces.symplectic}, the point is that the Tschirnhausen
bundle \(\mathcal{E}_\psi\) is a quotient of \(\mathcal{E}_Q\) and has rank
\(3\) and only degree \(2\) on lines in \(Q\).  Viewing \(Q =
\mathbf{OGr}(1,5)\) as in \ref{intro-homogeneous-spaces.orthogonal},
non-ampleness is due to the fact that the restriction of \(\psi\) to the
quadric surface \(S \subset \PP^3\) defined by \(x_0x_1 + x_2x_3 = 0\) is
the Frobenius morphism onto the hyperplane section \(S \cong Q \cap \mathrm{V}(x_4)\),
so that \(\mathcal{E}_\psi\rvert_S \cong \mathcal{E}_S\), and the latter is not
ample; in more geometric language, the leaves of the foliation associated with
\(\psi\) are quadric surfaces, and this obstructs ampleness of \(\mathcal{E}_\psi\),
whence \(\mathcal{E}_Q\).

To close this introduction, we indicate as to how the Frobenius--Tschirnhausen
bundle \(\mathcal{E}_X\) may be related to the tangent bundle of \(X\). First,
Kunz's theorem from \cite{Kunz} says that the Frobenius--Tschirnhausen bundle
detects nonsingularity: \(\mathcal{E}_X\) is locally free if and only if \(X\)
is regular. Next, positive characteristic foliation theory as in \cite{Ekedahl:Foliations}
identifies the Frobenius morphism \(F \colon X \to X\) as the quotient map for
the infinitesimal action of the tangent bundle \(\mathcal{T}_X\) on \(X\)
itself; as such, \(\mathcal{E}_X\) is dual to the sheaf of coinvariants for
the action of \(\mathcal{T}_X\) on \(\sO_X\). Finally, and in relation to this,
there exists a canonical filtration of \(F^*\mathcal{E}_X\) whose graded pieces
are certain (truncated) divided powers of \(\mathcal{T}_X\): see \cite[\S3]{Sun}.
This final point, especially, indicates that positivity of \(\mathcal{T}_X\)
contributes to that of \(\mathcal{E}_X\); the spirit of this work is a conviction
that the converse should also be true to some extent.

\smallskip
\noindent\textbf{Notation. ---}
Throughout, \(\kk\) denotes an algebraically closed field of characteristic
\(p > 0\), over which all schemes appearing are taken over. A \emph{variety}
refers to an integral scheme which is separated and of finite type over
\(\kk\). For a scheme \(X\) and a finite locally free \(\sO_X\)-module
\(\mathcal{E}\), we follow geometric conventions and write
\(\PP\mathcal{E} \coloneqq \Proj\Sym\mathcal{E}^\vee\) for the projective
bundle of lines in \(\mathcal{E}\); the pushforward along
\(\pi \colon \PP\mathcal{E} \to X\) of the relatively ample tautological line
bundle is therefore \(\pi_*\sO_{\PP\mathcal{E}}(1) = \mathcal{E}^\vee\). As
such, \(\mathcal{E}\) is \emph{ample} if the tautological line
bundle of \(\PP\mathcal{E}^\vee\) is ample.

\smallskip
\noindent\textbf{Acknowledgements. ---}
We would like to thank Jefferson Baudin, Javier Carvajal-Rojas, Alapan
Mukhopadhyay, Zsolt Patakfalvi, Quentin Posva, Domenico Valloni, and Yang Zhang
for helpful comments and conversations. RC would especially like to thank Pieter
Belmans and Sam Mundy for valuable discussions about generalized Grassmannian
and groups of type \(\mathrm{G}_2\).

\section{Frobenius--Tschirnhausen ampleness}\label{section-setup}
The central object of study in this work is the Tschirnhausen bundle
\(\mathcal{E}_X\) of the absolute Frobenius morphism \(F \colon X \to X\) of a
smooth projective variety \(X\). Explicitly, and more generally for each
\(e \geq 1\), define the \emph{\(e\)-th Frobenius cokernel} to be the
\(\sO_X\)-module
\[
\mathcal{B}_{X,e} \coloneqq \coker(F^{e,\#} \colon \sO_X \to F_*^e\sO_X).
\]
Since \(X\) is smooth, Kunz's theorem \cite{Kunz} implies that
\(\mathcal{B}_{X,e}\) is locally free of rank \(p^{ne} - 1\) where
\(n \coloneqq \dim X\). Dually, the \emph{\(e\)-th Frobenius--Tschirnhausen
bundle} or \emph{Frobenius-trace kernel} is
\[
\mathcal{E}_{X,e} \coloneqq
\mathcal{B}_{X,e}^\vee \cong
\ker(\tau_e \colon F^e_*(\omega_X^{\otimes 1-p^e}) \to \sO_X)
\]
where the second identification arises from Grothendieck duality along
the finite flat morphism \(F^e \colon X \to X\); here, \(\omega_X\) is
the dualizing line bundle of \(X\), and the canonical map \(\tau_e\) is
referred to as the \emph{\(p^e\)-th Frobenius-trace morphism}. In the case
\(e = 1\), write also \(\mathcal{B}_X \coloneqq \mathcal{B}_{X,1}\)
and \(\mathcal{E}_X \coloneqq \mathcal{E}_{X,1}\).

An initial question regarding positivity of the Frobenius--Tschirnhausen
bundles is: might there be a hierarchy of positivity conditions depending on
the sequence of indices \(e \geq 0\) for which \(\mathcal{E}_{X,e}\) is
positive? Happily, there is just one possibility:

\begin{Lemma}\label{setup-all-ample}
For a smooth projective variety \(X\), the following statements
are equivalent:
\begin{enumerate}
\item\label{setup-all-ample.just-1}
\(\mathcal{E}_X\) is ample;
\item\label{setup-all-ample.some-e}
\(\mathcal{E}_{X,e}\) is ample for some \(e \geq 1\); and
\item\label{setup-all-ample.all-e}
\(\mathcal{E}_{X,e}\) is ample for all \(e \geq 1\).
\end{enumerate}
The equivalence also holds with nefness in place of ampleness.
\end{Lemma}

\begin{proof}
It is easier to work with the dual statements and to show anti-ampleness of the
\(\mathcal{B}_{X,e}\).
Clearly, \ref{setup-all-ample.all-e} \(\Rightarrow\) \ref{setup-all-ample.some-e},
and \ref{setup-all-ample.some-e} \(\Rightarrow\) \ref{setup-all-ample.just-1}
was observed in \cite[Remark 5.2]{CRP}. So it remains to show that
\ref{setup-all-ample.just-1} \(\Rightarrow\) \ref{setup-all-ample.all-e}.
Suppose that \(\mathcal{B}_X = \mathcal{B}_{X,1}\) is anti-ample. As already
observed in \emph{loc.~cit.}, for any integer \(e > 0\), factoring
\(F^e = F \circ F^{e-1}\) provides a short exact sequence
\[
0 \to \mathcal{B}_X \to \mathcal{B}_{X,e} \to F_*\mathcal{B}_{X,e-1} \to 0.
\]
By induction on \(e\), we may suppose that \(\mathcal{B}_{X,e-1}\) is anti-ample.
Since \(F \colon X \to X\) is finite and surjective, \(F_*\mathcal{B}_{X,e-1}\)
is also anti-ample by \cite[Theorem 1.1]{BLNN}---although stated in
characteristic \(0\), their proof works in any characteristic. Thus
\(\mathcal{B}_{X,e}\) is an extension of anti-ample vector bundles and so
is anti-ample itself. The proof for nefness is entirely analogous.
\end{proof}

Positivity of \(\mathcal{E}_X\) propagates to positivity of Tschirnhausen
bundles of all finite purely inseparable surjective morphisms
\(\varphi \colon Y \to X\) from smooth varieties; for brevity, we also say that
\(Y\) is a \emph{purely inseparable covering} of \(X\). The main point is
that any such morphism factors through a power of the absolute Frobenius
morphism: there is a morphism \(\psi \colon X \to Y\) and integer \(e \geq 0\)
such that the maps \(\varphi \circ \psi\) and \(\psi \circ \varphi\) are the
\(p^e\)-power Frobenius morphisms of \(X\) and \(Y\), respectively;
the minimal such \(e\) is called the \emph{height} of \(\varphi\). With this,
we have:

\begin{Proposition}\label{setup-ample-insep-tschirnhausen}
A smooth projective variety \(X\) is Frobenius--Tschirnhausen ample if
and only if \(\mathcal{E}_\varphi\) is ample for any purely inseparable
covering \(\varphi\colon Y \to X\) by a smooth projective variety.
\end{Proposition}

\begin{proof}
Let \(\psi \colon X \to Y\) be a morphism such that \(\varphi \circ \psi\) is the
\(p^e\)-power Frobenius morphism of \(X\), and consider the short exact sequence
\[
0 \to \sO_Y \to \psi_*\sO_X \to \mathcal{B}_\psi \to 0.
\]
Pushing this along \(\varphi\), identifying \(\varphi_*\psi_*\sO_X = F_*^e\sO_X\), and passing
to cokernels of the canonical maps \(\sO_X \to \varphi_*\sO_Y\) and \(\sO_X \to F_*^e\sO_X\)
provides a short exact sequence of \(\sO_X\)-modules:
\[
0 \to \mathcal{B}_\varphi \to \mathcal{B}_{X,e} \to \varphi_*\mathcal{B}_\psi \to 0
\]
wherein the result follows from \parref{setup-all-ample} upon taking duals.
\end{proof}

\subsectiondash{Numerical restrictions}\label{setup-positivity}
Grothendieck--Riemann--Roch shows that
\[
c_1(\mathcal{E}_X) =
-c_1(F_*\sO_X) =
\frac{1}{2}p^{n-1}(p-1) c_1(\mathcal{T}_X)
\]
so that if \(\mathcal{E}_X\) is ample, then \(X\) must be Fano; see also
\cite[Proposition 5.9]{CRP}. This implies that, in dimension \(n = 1\),
\(X = \PP^1\) is the only curve with ample Frobenius-trace kernel. In higher
dimensions, by restricting to smooth rational curves, this provides a little
more. The next result is phrased for \(X\) of dimension \(n \geq 3\) as it
provides a slightly stronger statement. When \(n = 2\), the
argument shows that if \(X\) has a curve of anti-canonical degree \(\leq
2\), then \(\mathcal{E}_X\) is not ample. Since an exceptional curve on a del
Pezzo surface has degree \(1\), if \(\mathcal{E}_X\) is ample,
then \(X\) must be minimal. In the case of \(\PP^1 \times \PP^1\), a fibre of
one of the projections has degree \(2\), and so its
Frobenius-trace kernel is also not ample. Thus the only surface with ample
\(\mathcal{E}_X\) is \(X = \PP^2\), providing a simple
proof of \cite[Corollary 5.15]{CRP}.

\begin{Lemma}\label{setup-numerical-obstruction}
Let \(X\) be a smooth projective variety of dimension \(n \geq 3\) and suppose
that \(X\) contains a smooth rational curve on which \(c_1(\mathcal{T}_X)\) has degree \(k\).  If either
\(p = 2\) and \(k \leq 3\) or \(p \geq 3\) and \(k \leq 2\), then \(\mathcal{E}_X\)
is not ample.
\end{Lemma}

\begin{proof}
Prove the contrapositive: If \(\mathcal{E}_X\) were ample, then the degree \(k\)
of \(c_1(\mathcal{T}_X)\) on an arbitrary smooth rational curve \(C \subset X\) satisfies
\(k \geq 4\) if \(p = 2\) and \(k \geq 3\) if \(p \geq 3\). In this case,
\(\mathcal{E}_X\rvert_C\) is an ample vector bundle of rank
\(r \coloneqq p^n - 1\) and degree \(d \coloneqq \frac{1}{2}p^{n-1}(p-1)k\) by
the above calculation. Since vector bundles on \(C \cong \PP^1\) split into
a sum of line bundles, this implies that \(d \geq r\), and isolating \(k\)
leads to
\[
k
\geq \frac{2(p^n - 1)}{p^{n-1}(p-1)}
= 2 + 2\frac{p^{n-1} - 1}{p^{n-1}(p-1)}
\]
The second term is positive as soon as \(n \geq 2\), so integrality
implies \(k \geq 3\); if \(p = 2\) and \(n \geq 3\), the second term is greater
than \(1\), whence \(k \geq 4\).
\end{proof}

Applying this to complete intersections in projective space provides a simple
proof to a slight refinement of \cite[Theorem 5.1]{Mallory}:

\begin{Corollary}\label{setup-numerical-obstruction-complete-intersections}
Let \(X \subset \PP^n\) be a smooth complete intersection of dimension \(\geq 3\)
and total degree \(d\). If either \(p = 2\) and \(d \geq n-2\) or \(p \geq 3\)
and \(d \geq n-1\), then \(\mathcal{E}_X\) is not ample.
\end{Corollary}

\begin{proof}
This follows directly from \parref{setup-numerical-obstruction} since
\(c_1(\mathcal{T}_X) = (n+1-d)H\) for \(H = c_1(\sO_X(1))\) the hyperplane class
and any such \(X\) has a line, that is, a curve of \(H\)-degree \(1\): see
\cite{DM:Fano} or \cite[V.4.10]{Kollar:Curves}.
\end{proof}

The examples of homogeneous spaces as considered in
\S\parref{section-grassmannian} show that these simple numerical considerations
cannot be improved, and that \parref{setup-numerical-obstruction} is in a sense
optimal. Finer considerations are therefore required to obtain further
obstructions to ampleness of \(\mathcal{E}_X\).

\section{Subschemes}\label{sections-subschemes}
If \(\mathcal{E}_X\) were ample, so would its restriction to any closed
subscheme \(Z \subseteq X\). This method is made effective by Mallory's
observation from \cite{Mallory} that, dually, the kernel of the restriction map
\(F_*\sO_X\rvert_Z \to F_*\sO_Z\) looks like a conormal bundle, related to the
Frobenius neighbourhoood of \(Z\) in \(X\). In this section, we refine some of
Mallory's techniques to show in \parref{formal-G1} that subschemes of
Frobenius--Tschirnhausen ample varieties support only constant formal functions;
consequently, as explained in \parref{formal-finite-morphisms}, all
non-constant morphisms out of such varieties are finite onto their image.

\subsectiondash{Restriction to subschemes}\label{setup-subschemes}
Mallory observed in \cite[Lemma 3.2]{Mallory} that the Frobenius cokernel of a
scheme \(X\) and that of a closed subscheme \(Z\) defined by the ideal sheaf
\(\mathcal{I}\) are related via a short exact sequence of \(\sO_Z\)-modules of
the form
\[
0 \to
F_*^e(\mathcal{I}/\mathcal{I}^{[p^e]}) \to
\mathcal{B}_{X,e}\rvert_Z \to
\mathcal{B}_{Z,e} \to
0
\]
where \(\mathcal{I}^{[p^e]}\) is the ideal sheaf generated by \(p^e\)-powers of
sections of \(\mathcal{I}\). This sequence arises from the Frobenius direct
image of the sequence of \(\sO_X\)-modules
\[
0 \to \mathcal{I}/\mathcal{I}^{[p^e]} \to \sO_X/\mathcal{I}^{[p^e]} \to \sO_Z \to 0
\]
indicating that \(\mathcal{B}_{X,e}\rvert_Z\) carries information about certain
infinitesimal neighbourhoods of \(Z\) in \(X\).

In the case that \(X\) is Frobenius--Tschirnhausen ample, this provides strong
restrictions on the internal geometry of \(X\), one such being best phrased in
terms of the formal completion \(\hat{X}\) along a closed subscheme
\(Z \subset X\). Recall from \cite[Definition 2.9]{HM} that \(Z\) is said to be
\emph{G1} in \(X\) if the canonical map
\(\mathrm{H}^0(X,\sO_X) \to \mathrm{H}^0(\hat{X}, \sO_{\hat{X}})\)
is an isomorphism. With this, we have:

\begin{Theorem}\label{formal-G1}
Any connected closed subscheme of positive dimension in a smooth projective
Frobenius--Tschirnhausen ample variety \(X\) is G1 in \(X\).
\end{Theorem}

\begin{proof}
Let \(Z \subset X\) be a closed connected subscheme and let \(\hat{X}\) be the
completion of \(X\) along \(Z\).  Writing \(\mathcal{I}\) for its ideal sheaf,
note that the two sequences of ideal sheaves \(\{\mathcal{I}^k\}_{k \geq 0}\)
and \(\{\mathcal{I}^{[p^e]}\}_{e \geq 0}\) are mutually cofinal, so
\[
\mathrm{H}^0(\hat{X}, \sO_{\hat{X}}) =
\lim\nolimits_k \mathrm{H}^0(X, \sO_X/\mathcal{I}^k) =
\lim\nolimits_e \mathrm{H}^0(X, \sO_X/\mathcal{I}^{[p^e]}).
\]
Thus it suffices to show that
\(\mathrm{H}^0(X,\sO_X/\mathcal{I}^{[p^e]}) \cong \kk\) for all \(e \geq 0\).
Restricting the \(e\)-th Frobenius cokernel \(\mathcal{B}_{X,e}\) to \(Z\)
and using \parref{setup-subschemes} provides a short exact sequence
\[
0 \to
F_*^e(\mathcal{I}/\mathcal{I}^{[p^e]}) \to
\mathcal{B}_{X,e}\rvert_Z \to
\mathcal{B}_{Z,e} \to
0.
\]
Ampleness of the Frobenius-trace kernel implies via \parref{setup-all-ample}
that \(\mathcal{B}_{X,e}\rvert_Z\) is anti-ample, whence has no sections by
\parref{formal-antiample-bundles} below. Since Frobenius and any closed immersion
is affine,
\[
\mathrm{H}^0(X, \mathcal{I}/\mathcal{I}^{[p^e]}) =
\mathrm{H}^0(Z, F_*^e(\mathcal{I}/\mathcal{I}^{[p^e]})) \subseteq
\mathrm{H}^0(Z, \mathcal{B}_{X,e}\rvert_Z) = 0.
\]
So for each \(e \geq 0\), cohomology of the
sequence
\[
0 \to
\mathcal{I}/\mathcal{I}^{[p^e]} \to
\sO_X/\mathcal{I}^{[p^e]} \to
\sO_Z \to
0
\]
together with connectedness and projectivity of \(Z\) gives
\(\mathrm{H}^0(X,\sO_X/\mathcal{I}^{[p^e]}) \cong \mathrm{H}^0(Z,\sO_Z) \cong \kk\).
\end{proof}

In the proof, we used the fact that the dual of an ample vector
bundle on a positive-dimensional projective scheme has no sections.  For lack
of suitable reference, we include a proof:

\begin{Lemma}\label{formal-antiample-bundles}
Let \(X\) be a projective scheme over \(\kk\) and \(\mathcal{E}\) a finite
locally free \(\sO_X\)-module. If \(\dim X > 0\) and \(\mathcal{E}\) is ample,
then \(\Gamma(X,\mathcal{E}^\vee) = 0\).
\end{Lemma}

\begin{proof}
Suppose \(\sigma \colon \mathcal{E} \to \sO_X\) is the morphism dual to a
nonzero section of \(\mathcal{E}^\vee\). Since \(X\) is of positive dimension,
we may choose a finite morphism \(\nu \colon C \to X\) from a smooth projective
curve such that \(\nu^*\sigma \neq 0\): for instance, take \(C\) to be the
normalization of a component of a general complete intersection curve in \(X\).
The image of \(\nu^*\sigma \colon \nu^*\mathcal{E} \to \sO_C\) must then be a
line subbundle of \(\sO_C\), whence of degree \(\leq 0\).  This implies
\(\nu^*\mathcal{E}\) is not ample, and so neither is \(\mathcal{E}\) since
\(\nu\) is finite.
\end{proof}

A well-known consequence of \parref{formal-G1} is that any nontrivial morphism
out of \(X\) must be finite:

\begin{Corollary}\label{formal-finite-morphisms}
If \(X\) is a smooth projective variety which is Frobenius--Tschirnhausen
ample, then any non-constant \(f \colon X \to Y\) is finite onto its image.
\end{Corollary}

\begin{proof}
Let \(f \colon X \to Y\) be a non-constant morphism where the fibre
\(X_y \coloneqq f^{-1}(y)\) over some point \(y \in Y\) is
positive-dimensional; by replacing \(f\) by its Stein factorization, we may
assume that \(X_y\) is connected. Writing \(\hat{X}\) for the completion of
\(X\) along \(X_y\), the theorem on formal functions, as in \citeSP{02OD} for
instance, implies
\[
\mathrm{H}^0(\hat{X}, \sO_{\hat{X}}) \cong (f_*\sO_X)_y^{\wedge}.
\]
Since \(f\) is non-constant, \(Y\) has positive dimension, and so the completion
on the right contains a subring of the form \(\kk[\![t]\!]\); in particular,
it has infinite dimension as a \(\kk\)-vector space. Thus \(X_y\) is not
G1 in \(X\), so the Frobenius-trace kernel of \(X\) cannot be ample by
\parref{formal-G1}.
\end{proof}

Together with Mori's theory of extremal contractions, this strongly suggests
that Frobenius--Tschirnhausen ample varieties
have Picard rank \(1\). Since contractions exist for smooth Fano threefolds by
\cite{Kollar:Contractions}, this gives a direct proof of
\cite[Theorem 5.19]{CRP}: a Frobenius--Tschirnhausen ample threefold
is a Fano variety of Picard rank \(1\). Numerical considerations as in
\parref{setup-numerical-obstruction} together with the fact that Fano threefolds
of Picard rank \(1\) of index \(1\) or \(2\) contain anti-canonical conics
implies that the only Frobenius--Tschirnhausen ample threefolds are \(\PP^3\)
and, when \(p \neq 2\), a smooth quadric threefold:
see \cite[Theorem 4.1]{Mallory} for details.

\subsectiondash{Filtration}\label{setup-filtration}
A more computable obstruction to Frobenius--Tschirnhausen
ampleness can be extracted from \parref{setup-subschemes} by considering the
\(\mathcal{I}\)-adic filtration of the \(\sO_X\)-module
\(\mathcal{I}/\mathcal{I}^{[p]}\):
\[
\mathcal{I}^k \cdot \mathcal{I}/\mathcal{I}^{[p]} =
\mathcal{I}^{k+1}/\mathcal{I}^{k+1} \cap \mathcal{I}^{[p]}
\;\;\text{for each \(k \geq 0\)}.
\]
Since \(\mathcal{I}^k \subseteq \mathcal{I}^{[p]}\) for \(k \gg 0\), this
provides a separated exhaustive filtration of
\(\mathcal{I}/\mathcal{I}^{[p]}\). If \(Z\) is a local complete
intersection subscheme of \(X\) of codimension \(c\), \cite[Proposition
3.4]{Mallory} provides a neat description of the graded pieces of this
filtration in terms of the conormal bundle \(\mathcal{C}_{Z/X} \coloneqq
\mathcal{I}/\mathcal{I}^2\):
\begin{align*}
\gr^k(\mathcal{I}/\mathcal{I}^{[p]})
& \coloneqq (\mathcal{I}^k \cdot \mathcal{I}/\mathcal{I}^{[p]})/(\mathcal{I}^{k+1} \cdot \mathcal{I}/\mathcal{I}^{[p]}) \\
& \cong \coker(\Sym^{k+1-p}\mathcal{C}_{Z/X} \otimes F^*\mathcal{C}_{Z/X} \to \Sym^{k+1}\mathcal{C}_{Z/X}) \eqqcolon
T^{k+1}(\mathcal{C}_{Z/X})
\end{align*}
where \(T^{k+1}\) is sometimes called the \emph{\((k+1)\)-st truncated
symmetric power functor}; here, the displayed map is multiplication and
negative symmetric powers are taken to zero.
Notably, the last nonzero piece of the \(\mathcal{I}\)-adic filtration provides
a submodule of \(\mathcal{I}/\mathcal{I}^{[p]}\) of the form
\[
\det(\mathcal{C}_{Z/X})^{\otimes p-1} \cong
T^{c(p-1)}(\mathcal{C}_{Z/X}) \cong
\gr^{c(p-1)-1}(\mathcal{I}/\mathcal{I}^{[p]}) =
\mathcal{I}^{c(p-1)}/\mathcal{I}^{c(p-1)} \cap \mathcal{I}^{[p]} \subseteq \mathcal{I}/\mathcal{I}^{[p]}.
\]
Combined with the exact sequence from \parref{setup-subschemes} and
\parref{formal-antiample-bundles}, this gives Mallory's criterion \cite[Theorem
1.1]{Mallory}: if there exists a local complete intersection subscheme
\(Z \subset X\) such that \(\det(\mathcal{C}_{Z/X})^{\otimes p-1}\) has a
section, then \(X\) is not Frobenius--Tschirnhausen ample. This may be used to
show that, for example, cubic fourfolds containing planes are not
Frobenius--Tschirnhausen ample.

Other steps of this filtration may be effectively used to obstruct
Frobenius--Tschirnhausen ampleness when the filtration splits. The following
identifies a special class of Fano hypersurfaces in projective space for
which this method may be applied:

\begin{Proposition}\label{formal-special-hypersurfaces}
A smooth Fano hypersurface \(X \subset \PP^n\) of degree \(d > p\)
projectively equivalent to
\[
(x_0^{d-1} x_n + x_1^{d-1} x_{n-1} + x_2^p F_2 + \cdots + x_n^p F_n = 0) \subset \PP^n
\]
for some \(F_2,\ldots,F_n \in \mathrm{H}^0(\PP^n, \sO_{\PP^n}(d-p))\)
is not Frobenius--Tschirnhausen ample.
\end{Proposition}

\begin{proof}
Change coordinates and assume that \(X\) is defined by the displayed equation.
Then it contains the line \(\ell\) whose ideal sheaf \(\mathcal{I}\) is
generated by the coordinates \(x_2,\ldots,x_n\). We show \(X\) is not
Frobenius--Tschirnhausen ample by showing that
\(\mathrm{H}^0(X,F_*(\mathcal{I}/\mathcal{I}^{[p]})) \neq 0\), at which point
the exact sequence of \parref{setup-subschemes} together with
\parref{formal-antiample-bundles} yields the conclusion. The observation is that,
due to the form of the equation, the \(\mathcal{I}\)-adic filtration on
\(\mathcal{I}/\mathcal{I}^{[p]}\) from \parref{setup-filtration} splits upon
viewing it as an \(\sO_\ell\)-module
through \(F_*\), so there is an isomorphism of \(\sO_\ell\)-modules
\[
F_*(\mathcal{I}/\mathcal{I}^{[p]}) \cong
\bigoplus\nolimits_{k = 1}^{(n-2)(p-1)} F_*T^k(\mathcal{C}_{\ell/X}).
\]
Indeed, trivialize this bundle on the two standard affine open subschemes of
\(\ell\) given by the nonvanishing locus of \(x_0\) and \(x_1\); for example,
on the chart where \(x_0 = 1\), the equation of \(X\) implies
\[
x_n
= -x_1^{d-1}x_{n-1} - x_2^pF_2 - \cdots - x_n^pF_n
\equiv -x_1^{d-1}x_{n-1} + \mathcal{I}^{[p]}
\]
so that an \(\sO_\ell\)-module basis for the restriction of
\(F_*(\mathcal{I}/\mathcal{I}^{[p]})\) thereon is given by the monomials
\[
x_1^a x_2^{b_2} \cdots x_{n-2}^{b_{n-2}} x_{n - 1}^{b_{n-1}}
\;\;\text{with}\; 0 \leq a,b_2,\ldots,b_{n-2},b_{n-1} \leq p-1
\]
and \(b_2 + \cdots + b_{n-2} + b_{n-1} > 0\). Passing to the chart where
\(x_1 = 1\) transforms the displayed basis element to a similar one
where \(x_1\) and \(x_{n-1}\) are replaced by \(x_0\) and \(x_n\); more
precisely, if \(m\) and \(0 \leq a' \leq p-1\) are the integers determined
by the relation
\[
mp + a' = -(a + b_2 + \cdots + b_{n-2}) + (d-2)b_{n-1},
\]
then
\(
x_1^a x_2^{b_2} \cdots x_{n-2}^{b_{n-2}} x_{n - 1}^{b_{n-1}}
\mapsto 
x_0^m \cdot x_0^{a'} x_2^{b_2} \cdots x_{n-2}^{b_{n-2}} x_n^{b_{n-1}}
\); note that the \(\sO_\ell\)-module structure is through \(p\)-th powers, so
that \(x_0 \cdot s = x_0^ps\) for a local section \(s\) of
\(F_*(\mathcal{I}/\mathcal{I}^{[p]})\). The transition function in the
given basis is block diagonal and preserves the exponents
\(b_2,\ldots,b_{n-2},b_{n-1}\), and this implies the claimed splitting.
Finally, by considering the basis elements where
\(b_2 + \cdots + b_{n-2} + b_{n-1} = 1\) or else by a direct computation, the
conormal bundle may be determined to be
\[
\mathcal{C}_{\ell/X}
= \mathcal{I}/\mathcal{I}^2
\cong \sO_{\ell}(-1)^{\oplus n-3} \oplus \sO_{\ell}(d-2).
\]
This has sections, whence the result follows.
\end{proof}

\subsectiondash{}\label{formal-special-hypersurfaces-exist}
To complete the proof of Theorem \parref{intro-nonample-hypersurfaces}, it
remains to see that there are smooth hypersurfaces whose equation takes the
form in \parref{formal-special-hypersurfaces}. Suppose first that \(p \nmid d\)
and that \(n = 2m-1\) is odd. Consider the hypersurface
\(X \subset \PP^{2m-1}\) of degree \(d\) defined by
\[
f \coloneqq
\begin{dcases*}
\sum\nolimits_{k = 0}^{m-1} (x_k^{d-1} x_{n-k} + x_k x_{n-k}^{d-1}) & if \(d \equiv 1 \;(\text{mod}\;p)\), and \\
\sum\nolimits_{k = 0}^{m-1} (x_k^{d-1} x_{n-k} + x_{n-k}^d) & otherwise.
\end{dcases*}
\]
By the Jacobian criterion, a singular point of \(X\) satisfies,
for each \(k = 0,\ldots,m-1\):
\[
0 =
\frac{\partial f}{\partial x_k} =
\begin{dcases*}
x_{n-k}^{d-1} \\
(d-1)x_k^{d-2}x_{n-k}
\end{dcases*}
\;\;\text{and}\;\;
0 =
\frac{\partial f}{\partial x_{n-k}} =
\begin{dcases*}
x_k^{d-1} & if \(d \equiv 1 \;(\text{mod}\; p)\), \\
x_k^{d-1} + dx_{n-k}^{d-1} & otherwise.
\end{dcases*}
\]
In either case, the only solution to all these equations is that given by
\(x_k = 0\) for all \(k = 0,\ldots,2m-1\),  so \(X\) is smooth. When
\(p \nmid d\) and \(n = 2m\) is even, let \(f\) be defined by the same formula
as above and simply take \(X\) to be the hypersurface defined by \(f + x_m^d\).
Next, suppose that \(p \mid d\), and consider
\[
f \coloneqq
x_0^{d-1} x_n +
x_n^{d-1} x_{n-2} +
x_{n-2}^{d-1} x_{n-3} +
\cdots +
x_2^{d-1} x_1 +
x_1^{d-1} x_{n-1} +
x_{n-1}^{d-1} x_0.
\]
Up to reindexing the coordinates, this equation generalizes that defining
the Klein cubic hypersurface; it is shown in \cite[Lemma 4]{AGR} that when
\(p \nmid n+1\), the hypersurface \(X \subset \PP^n\) cut out by \(f\)
is smooth. Finally, when \(p \mid n+1\), take \(X\) to be defined by
\(f + x_n^d\); then a minor variation on the argument of \emph{loc. cit.} shows
that \(X\) is smooth.
\qed

\section{Homogeneous spaces}\label{section-homogeneous}
Frobenius-trace kernels of projective rational homogeneous spaces tend to be
positive. To describe this, fix notation and write throughout this and the
following sections \(\mathbf{T} \subseteq \mathbf{P} \subseteq \mathbf{G}\) for
a maximal torus and reduced parabolic subgroup contained in a connected
reductive algebraic group, and let \(X \coloneqq \mathbf{G}/\mathbf{P}\) be the
associated projective rational homogeneous space. The first observation,
recorded in \parref{homogeneous-globally-generated}, is that Frobenius splitness
of \(X\) combined with the infinitesimal representation theory of the
opposite unipotent radical \(\mathbf{U}\) of \(\mathbf{P}\) implies that
the homogeneous vector bundle \(F_*(\omega_X^{\otimes 1-p})\), and hence
also the Frobenius-trace kernel \(\mathcal{E}_X\), is globally generated.

\subsectiondash{Global generation}\label{homogeneous-local-computations}
To explain, observe that \(F_*(\omega_X^{\otimes 1-p})\) and
\(\mathcal{E}_X\), being constructed from canonical operations on \(X\), are
homogeneous vector bundles, so they are globally generated if and only if
global sections generate them at the point \(o \in X\) corresponding to the
identity of \(\mathbf{G}\). Since \(X\) is Frobenius split
by \cite[Theorem 2]{MR:F-Splitting}, the trace sequence
\[
0 \to
\mathcal{E}_X \to
F_*(\omega_X^{\otimes 1-p}) \to
\sO_X \to
0
\]
shows that \(\mathcal{E}_X\) is generated at \(o\) if and only if
\(F_*(\omega_X^{\otimes 1-p})\) is so. Since \(F \colon X \to X\) is an
affine morphism, the latter statement is equivalent to surjectivity of the
restriction map
\[
\res \colon
\mathrm{H}^0(X,\omega_X^{\otimes 1-p}) \to
\mathrm{H}^0(X,\omega_X^{\otimes 1-p} \otimes \sO_X/\mathfrak{m}^{[p]})
\]
where \(\mathfrak{m} \subset \sO_X\) is the ideal sheaf of \(o \in X\);
in the terminology of \cite{MS:F-Seshadri}, this means that \(\mathcal{E}_X\)
is globally generated if and only if \(\omega_X^{\otimes 1-p}\) \emph{separates
\(p\)-Frobenius jets} at \(o\). 

To analyze this restriction map, decompose the parabolic subgroup
\(\mathbf{P} = \mathbf{U}^+ \rtimes \mathbf{L}\) into its unipotent radical and
a Levi factor, and let \(\mathbf{U} \subset \mathbf{G}\) be the opposite
unipotent radical. The canonical map
\(\mathbf{U} \to \mathbf{G}/\mathbf{P} = X\) is an open immersion, equivariant
with respect to multiplication by \(\mathbf{U}\) and conjugation by
\(\mathbf{L}\), and such that the Frobenius kernel
\(\mathbf{U}_1 \coloneqq \ker(F \colon \mathbf{U} \to \mathbf{U})\) is mapped
onto the subscheme \(F^{-1}(o)\). The restriction map may therefore be
identified as a map of \(\mathbf{U}_1 \rtimes \mathbf{L}\)-representations
\[
\res \colon
\mathrm{H}^0(X,\omega_X^{\otimes 1-p}) \to
\kk_{-\lambda} \otimes \kk[\mathbf{U}_1]
\]
where \(\kk[\mathbf{U}_1]\) is the coordinate ring of the finite infinitesimal
group scheme \(\mathbf{U}_1\), and \(\kk_{-\lambda}\) is the \(1\)-dimensional
\(\mathbf{L}\)-representation whose character \(-\lambda\) is the restriction of
the \(\mathbf{P}\)-character corresponding to the line bundle dual to
\(\omega_X^{\otimes 1-p}\) under the usual identification of \(\operatorname{Pic} X\)
with the character group of \(\mathbf{P}\).

\subsectiondash{}\label{homogeneous-local-computations-coordinates}
To describe the coordinate ring \(\kk[\mathbf{U}_1]\) explicitly, let
\(\mathrm{X}^*(\mathbf{T}) \supset \Phi \supset \Phi^+ \supset \Phi^+_{\mathbf{P}}\)
be, respectively, the character lattice of the maximal torus, the roots of
\(\mathbf{G}\), positive roots compatible with \(\mathbf{P}\), and the roots
appearing in \(\mathbf{U}^+\). Choose an isomorphism of varieties
\[
\mathbf{U} \cong
\prod\nolimits_{\alpha \in \Phi^+_{\mathbf{P}}} \mathbf{U}_{-\alpha}
\]
given by an ordered product of root subgroups: see \cite[II.1.2 and
II.1.10]{Jantzen:RAGS}. Since
\(\mathbf{U}_{-\alpha} \cong \mathbf{G}_a\) for each \(\alpha\), this
identifies \(\mathbf{U}\) as an affine space with coordinates \(x_\alpha\)
for each \(\alpha \in \Phi^+_{\mathbf{P}}\). In this way, the coordinate
ring of \(\mathbf{U}_1 \cong F^{-1}(o)\) may be identified as the Artinian
ring
\[
\kk[\mathbf{U}_1] \cong
\mathrm{H}^0(X, \sO_X/\mathfrak{m}^{[p]}) \cong
\kk[x_\alpha : \alpha \in \Phi^+_{\mathbf{P}}]/(x_\alpha^p : \alpha \in \Phi^+_{\mathbf{P}}).
\]
Under any such isomorphism, the monomials span weight spaces for
\(\mathbf{T} \subseteq \mathbf{L}\) where a single coordinate \(x_\alpha\)
is of weight \(\alpha\). There is a unique monomial of highest weight given by
\[
\tau \coloneqq \Big(\prod\nolimits_{\alpha \in \Phi^+_{\mathbf{P}}} x_\alpha\Big)^{p-1}
\]
with weight
\(\lambda = (p-1)\sum\nolimits_{\alpha \in \Phi^+_{\mathbf{P}}} \alpha\).
Since \(\mathbf{U}_1\) is unipotent, its action on \(\kk[\mathbf{U}_1]\) cannot
increase weights, and this implies that the complementary
\(\mathbf{T}\)-subspace \(\kk[\mathbf{U}_1]_{< \lambda}\) spanned by all
monomials other than \(\tau\) is a \(\mathbf{U}_1\)-stable hyperplane; in fact,
it is the unique \(\mathbf{U}_1\)-stable hyperplane by \cite[Lemma
5.4]{Kempf:RAGS}. This observation leads to:

\begin{Proposition}\label{homogeneous-globally-generated}
The Frobenius-trace kernel of a rational homogeneous space is globally
generated.
\end{Proposition}

\begin{proof}
By the discussion of \parref{homogeneous-local-computations}, it suffices to
show that \(\omega_X^{\otimes 1-p}\) separates \(p\)-Frobenius jets at
\(o \in X\). Upon choosing an isomorphism of \(\mathbf{U}\) with a product
of root groups as in \parref{homogeneous-local-computations-coordinates},
view the restriction map as a map of \(\mathbf{U}_1 \rtimes \mathbf{L}\)-representations
\[
\res \colon
\mathrm{H}^0(X,\omega_X^{\otimes 1-p}) \to
\kk_{-\lambda} \otimes \kk[x_\alpha : \alpha \in \Phi^+_{\mathbf{P}}]/
(x_\alpha^p : \alpha \in \Phi^+_{\mathbf{P}}).
\]
If \(\res\) were not surjective, the discussion of
\parref{homogeneous-local-computations-coordinates} implies that its image must
be contained in the hyperplane
\(\kk_{-\lambda} \otimes \kk[\mathbf{U}_1]_{< \lambda}\). But \(\tau\) is in
the image of \(\res\) since \(X\) is Frobenius split: see
\cite[Theorem 1.3.8]{BK:F-Splitting}. Therefore \(\res\) is surjective, and so
\(\mathcal{E}_X\) is globally generated.
\end{proof}

Although \(\mathcal{E}_X\) is always globally generated for
\(X = \mathbf{G}/\mathbf{P}\), it certainly not
always ample: If \(\mathbf{P}\) were contained in a larger
parabolic \(\mathbf{P}' \subsetneq \mathbf{G}\), then there is a
fibration \(X \to \mathbf{G}/\mathbf{P}'\) with fibres
\(\mathbf{P}'/\mathbf{P}\), and \parref{formal-finite-morphisms} implies
\(\mathcal{E}_X\) is not ample. Ampleness thus forces \(\mathbf{P}\) to be a
\emph{maximal} parabolic subgroup; geometrically, this means \(X\) must be a
\emph{generalized Grassmannian}, that is, a rational homogeneous
space of of Picard rank \(1\).  Unfortunately, it is still too optimistic to
hope that \(\mathcal{E}_X\) is always ample in this case, and the following
example shows that some hypothesis on the characteristic \(p\) is required:

\subsectiondash{Example}\label{homogeneous-quadric-threefold}
A smooth quadric hypersurface \(X \subset \PP^4\), a homogeneous space for
\(\mathbf{G} = \mathbf{SO}_5\), has ample Frobenius-trace kernel \(\mathcal{E}_X\)
if and only if \(p > 2\). When \(p > 2\), \cite[Theorem 2]{Achinger:Quadrics}
shows that all the summands appearing in \(\mathcal{E}_X\) are ample, whereas
for \(p = 2\), Theorem 1 of \emph{loc. cit.} gives
\[
\mathcal{E}_X \cong \sO_X(1)^{\oplus 5} \oplus \mathcal{S}^\vee
\]
where the dual spinor bundle \(\mathcal{S}^\vee\) is not ample. More directly,
this also follows from
\parref{setup-numerical-obstruction-complete-intersections}.
\qed
\smallskip

Nevertheless, we conjecture that, up to excluding low characteristic
exceptions, the Frobenius-trace kernel \(\mathcal{E}_X\) of a generalized
Grassmannian is ample. In the remainder of the section, we refine the
techniques used in \parref{homogeneous-globally-generated} to develop a method
of showing the stronger property that \(\mathcal{E}_X \otimes \sO_X(-1)\) is
globally generated, where \(\sO_X(1)\) is the ample generator of \(\Pic X\).

\subsectiondash{Ampleness}\label{homogeneous-ampleness-computation}
Let \(X\) be a generalized Grassmannian, write \(\sO_X(1)\) for the ample
generator of \(\Pic X\), and consider the twisted Frobenius-trace sequence
\[
0 \to
\mathcal{E}_X \otimes \sO_X(-1) \to
F_*(\omega_X^{\otimes 1 - p} \otimes \sO_X(-p)) \to
\sO_X(-1) \to
0.
\]
The considerations of \parref{homogeneous-local-computations-coordinates} and
\parref{homogeneous-globally-generated} imply that, at \(o \in X\),
the trace map to \(\sO_X(-1)\) is canonically identified as the projection onto
the highest weight space of \(\kk[\mathbf{U}_1]\) spanned by the element
\(\tau\). Since \(\sO_X(-1)\) has no sections, the restriction map from
global sections of \(\omega_X^{\otimes 1-p} \otimes \sO_X(-p)\) to
\(p\)-Frobenius jets at \(o\) factors through the hyperplane
\[
\res' \colon
\mathrm{H}^0(X, \omega_X^{\otimes 1-p} \otimes \sO_X(-p)) \to
\kk_{-\lambda + p\varpi} \otimes \kk[\mathbf{U}_1]_{< \lambda}
\]
where \(\varpi\) is the restriction to \(\mathbf{L}\) of the
weight corresponding to \(\sO_X(1)\). Arguing as in
\parref{homogeneous-local-computations} shows that \(\mathcal{E}_X \otimes \sO_X(-1)\)
is globally generated if and only if \(\res'\) is surjective. Observe that such
a statement is optimal: the line bundle \(\omega_X^{\otimes 1-p} \otimes
\sO_X(1-p)\) separates \(p\)-Frobenius jets because \(X\) is Frobenius split
compatibly with the Schubert divisor \(X \setminus \mathbf{U}\) opposite to
\(\mathbf{U}\): see \cite[Theorems 2.2.5 and 1.4.10]{BK:F-Splitting}.

\subsectiondash{}\label{homogeneous-ampleness-lie-algebra}
To formulate a suitable surjectivity criterion for \(\res'\), rephrase the
observation recalled in \parref{homogeneous-local-computations-coordinates} as
follows: \(\kk[\mathbf{U}_1]\) itself is the only \(\mathbf{U}_1\)-submodule
containing the trace element \(\tau\). Applying the Demazure--Gabriel
correspondence from \cite[II, \S7 n\textdegree 4]{DG}---see also
\cite[I.9.6]{Jantzen:RAGS}---which gives an equivalence between the
representation theory of the height \(1\) infinitimesmal group scheme
\(\mathbf{U}_1\) and its \(p\)-Lie algebra
\(\mathfrak{u} = \operatorname{Lie} \mathbf{U}_1 = \operatorname{Lie} \mathbf{U}\),
this equivalently says that every element of \(\kk[\mathbf{U}_1]\) may be
expressed as a linear combination of elements obtained by successively
acting by \(\mathfrak{u}\) on \(\tau\):
\[
\kk[\mathbf{U}_1] =
\operatorname{span} \{\tau\} +
\sum\nolimits_{k \geq 1}
\operatorname{span}\{X_{i_1} \cdots X_{i_k} \cdot \tau : X_{i_1}, \ldots, X_{i_k} \in \mathfrak{u}\}.
\]
Since the action of \(\mathfrak{u}\) decreases weights, the sum on the right
must be the hyperplane \(\kk[\mathbf{U}_1]_{< \lambda}\). This gives the following
generalization of Kempf's criterion: \(\kk[\mathbf{U}_1]_{< \lambda}\) is itself
the only \(\mathbf{U}_1\)-submodule that contains the image
\(\mathfrak{u} \cdot \tau\) of the trace element under the action
of its Lie algebra. The following makes this more explicit, and also gives a
geometric consequence; for simplicity, we suppress the global weight shift by
\(-\lambda + p\varpi\) in the statement:

\begin{Proposition}\label{homogeneous-ampleness-criterion}
The twisted Frobenius-trace kernel \(\mathcal{E}_X \otimes \sO_X(-1)\)
is globally generated if and only if the image of \(\res'\) contains
for each \(\alpha \in \Phi^+_{\mathbf{P}}\) the element
\[
-\partial_\alpha \tau
\coloneqq
  x_{\alpha}^{p-2}
  \Big(
    \prod\nolimits_{\beta \in \Phi^+_{\mathbf{P}} \setminus \alpha}
    x_{\beta}
  \Big)^{p-1}
\in
\kk[\mathbf{U}_1].
\]
In particular, in this case, \(X\) is Frobenius split compatibly with the
Schubert divisor \(X \setminus \mathbf{U}\) opposite to \(\mathbf{U}\) and
relative to each root hyperplane of \(\mathbf{U}\).
\end{Proposition}

\begin{proof}
Following the discussion of \parref{homogeneous-ampleness-computation} and
\parref{homogeneous-ampleness-lie-algebra}, it suffices to show that the action
on \(\tau\) of the element \(X_\alpha \in \mathfrak{u}\) dual to the root
coordinate \(x_\alpha\) is, up to scalars, by partial differentiation with
respect to \(x_\alpha\). By construction, \(X_\alpha\) acts on
\(\kk[\mathbf{U}_1]\) by derivations of weight \(-\alpha\); since
\(\mathbf{U}_{-\alpha} \cong \mathbf{G}_a\), it acts on \(x_\alpha\) by its
dual derivation \(\partial_{\alpha}\). Thus it may be expressed as an operator
of the form
\[
X_\alpha =
\partial_{\alpha} +
\sum\nolimits_{\beta \in \Phi^+_{\mathbf{P}} \setminus \alpha}
f_{\beta} \partial_\beta
\]
for some \(f_\beta \in \kk[\mathbf{U}_1]\) of weight \(\beta - \alpha\).
Weight considerations imply that any nonzero term appearing in
\(f_\beta\) must contain a root coordinate \(x_\gamma\) other than
\(x_\beta\); since \(x_{\gamma}^p = 0\) in \(\kk[\mathbf{U}_1]\), this implies
that \(f_\beta \partial_\beta \tau = 0\) for each \(\beta\), yielding the
first statement. For the second statement, simply note that the linear
function \(x_\alpha\) defining the \(\alpha\) root hyperplane of \(\mathbf{U}\)
is a section of \(\sO_X(1)\), so that \(x_\alpha \partial_\alpha \tau\)
is a section of \(\omega_X^{\otimes 1-p} \otimes \sO_X(1-p)\) that provides
a Frobenius splitting relative to the root hyperplane defined by \(x_\alpha\)
and compatible with the Schubert divisor \(X \setminus \mathbf{U}\), as per the
remark closing \parref{homogeneous-ampleness-computation}.
\end{proof}

To make this criterion more effective, notice that since \(\mathbf{L}\)
normalizes \(\mathbf{U}_1\), the action of \(\mathbf{L}\) on
\(\kk[\mathbf{U}_1]\) intertwines the action of the Lie algebra
\(\mathfrak{u}\), where \(\mathbf{L}\) acts on \(\mathfrak{u}\) via the adjoint
representation. This gives the following structure:

\begin{Corollary}\label{homogeneous-ampleness-u}
The \(\partial_\alpha\tau\) with \(\alpha \in \Phi^+_{\mathbf{P}}\)
span an \(\mathbf{L}\)-submodule of \(\kk[\mathbf{U}_1]\), isomorphic to
\[
\pushQED{\qed}
\operatorname{span}\{\partial_\alpha\tau : \alpha \in \Phi^+_{\mathbf{P}}\}
\cong \mathfrak{u} \cdot \tau
\cong \mathfrak{u} \otimes \kk_{\lambda}.
\qedhere
\popQED
\]
\end{Corollary}

\subsectiondash{Levels}\label{homogeneous-level}
More precise structure on the representations of \(\mathbf{L}\) will prove
useful. For this, let \(\Delta \subset \Phi^+\) be a set of simple
roots, and let \(\Delta_{\mathbf{L}} \subset \Delta\) be the subset of simple
roots in \(\mathbf{L}\). Given a weight \(\beta \in \mathrm{X}^*(\mathbf{T})\)
in the cone spanned by \(\Phi^+\), write
\[
\beta =
\sum\nolimits_{\alpha \in \Delta_{\mathbf{L}}} c_\alpha \alpha +
\sum\nolimits_{\alpha \in \Delta \setminus \Delta_{\mathbf{L}}} d_\alpha \alpha
\;\;\text{for integers}
\; c_\alpha, d_\alpha \geq 0.
\]
The integer \(\sum_\alpha d_\alpha\) is the \emph{level} of \(\beta\), and the
tuple \((d_\alpha)_{\alpha \in \Delta \setminus \Delta_{\mathbf{L}}}\) is
the \emph{shape} of \(\beta\). Since the Levi action shifts weights by
combinations of \(\Delta_{\mathbf{L}}\), levels are preserved, so any
representation of \(\mathbf{L}\) has a direct sum decomposition into levels. In
particular, this applies to its adjoint representation on \(\mathfrak{u}\) and
also its action on the coordinate ring of \(\mathbf{U}_1\); in the latter case,
this provides a multiplicative grading and the level \(\ell \geq 0\) piece may
be described in terms of the root coordinates \(x_\alpha\) from
\parref{homogeneous-local-computations-coordinates} as
\[
\kk[\mathbf{U}_1]_\ell \coloneqq
\operatorname{span}\{
x_{\alpha_1} \cdots x_{\alpha_k} :
\alpha_1,\ldots,\alpha_k \in \Phi^+_{\mathbf{P}}
\;\text{with}\;
\operatorname{level}(\alpha_1 + \cdots + \alpha_k) = \ell\}.
\]
In particular, the constants span level \(0\) and the trace element \(\tau\)
spans the top level \(N\). The elements \(\partial_\alpha\tau\) of interest
from \parref{homogeneous-ampleness-criterion} are typically spread across
several levels near the top of \(\kk[\mathbf{U}_1]\) depending, by
\parref{homogeneous-ampleness-u}, on the level decomposition of \(\mathfrak{u}\).
See \cite{ABS} and \cite[\S17.1]{MT:LAG} for more.

\section{Isotropic Grassmannian}\label{section-grassmannian}
In this section, we apply the methods developed in
\S\parref{section-homogeneous} to show that the twisted Frobenius-trace kernel
is globally generated for most generalized Grassmannian
\(X = \mathbf{G}/\mathbf{P}\). In particular, this implies that
\(\mathcal{E}_X\) is ample for such homogeneous spaces, yielding sufficiency in
Theorem \parref{intro-homogeneous-spaces}. Our method relies on some explicit
computations and hence works best on \emph{isotropic Grassmannian}, by which we
mean any \(X\) that may be described as a subvariety of the usual Grassmannian
parameterizing linear spaces isotropic for some additional structure. Most
familiar are those of \emph{classical type}:
\begin{itemize}
\item[(A)] the Grassmannian \(\Gr(k,n)\) of \(k\)-dimensional linear subspaces in an
\(n\)-dimensional vector space, where \(n \geq 2\) and \(1 \leq k \leq n-1\);
\item[(C)] the symplectic Grassmannian \(\SGr(k,n)\) parameterizing
totally isotropic subspaces for a given nondegenerate symplectic form
\(\omega\), where \(n = 2m \geq 4\) and \(1 \leq k \leq m\); and
\item[(B/D)] the orthogonal Grassmannian \(\OGr(k,n)\) parameterizing
totally isotropic subspaces for a given nondegenerate quadratic form \(q\),
where \(n \geq 5\) and \(1 \leq k < \lfloor\frac{n-1}{2}\rfloor\) or
\(k = \lfloor \frac{n}{2} \rfloor\).
\end{itemize}
In terms of Dynkin diagrams, these are of types \(\mathrm{A}_{n-1}\),
\(\mathrm{C}_m\), and, in the orthogonal case, \(\mathrm{B}_m\) if \(n = 2m+1\)
and \(\mathrm{D}_m\) if \(n = 2m\). We also study the two generalized
Grassmannian of type \(\mathrm{G}_2\): one being a quadric fivefold, and the
other---referred to as the \emph{\(\mathrm{G}_2\)-Grassmannian}---is a
\(5\)-dimensional subvariety of \(\Gr(2,7)\) parameterizing subspaces totally
isotropic for a generic alternating \(3\)-form. We refer to \cite{grassmannian}
for an interactive guide to these objects. With this notation, the result is:

\begin{Theorem}\label{homogeneous-theorem}
Let \(X\) be an isotropic Grassmannian. Then \(\mathcal{E}_X(-1)\) is
globally generated except possibly when:
\(X = \mathbf{SGr}(k,2m)\) with \(2 \leq k \leq m\) for \(p = 2\);
\(X = \mathbf{OGr}(m-1,2m+1)\) for \(p = 2\); and
\(X\) is the \(\mathrm{G}_2\)-Grassmannian for \(p = 2, 3\).
\end{Theorem}

\subsectiondash{}\label{homogeneous-strategy}
The proof of \parref{homogeneous-theorem} occupies the remainder of this section.
To explain the strategy, coordinatize the canonical affine
open subscheme of \(X\) provided by the opposite unipotent radical
\(\mathbf{U}\) as in \parref{homogeneous-local-computations} 
as the big Schubert cell parameterizing isotropic \(k\)-dimensional subspaces
in an \(n\)-dimensional space given by the row span of the \(k \times n\)
matrix of indeterminates
\[
M \coloneqq
\begin{pmatrix}
x_{1,1} & \cdots & x_{1,k} & x_{1,k+1} & \cdots & x_{1,n-k} & 1 & \cdots & 0 \\
\vdots  & \ddots & \vdots  & \vdots    & \ddots & \vdots    & \vdots & \ddots & \vdots \\
x_{k,1} & \cdots & x_{k,k} & x_{k,k+1} & \cdots & x_{k,n-k} & 0 & \cdots & 1
\end{pmatrix}
\]
where, other than in type \(\mathrm{A}\), the \(x_{i,j}\) are subject to
equations which express the isotropicity condition. With the exception
of the \emph{spinor varieties} dealt with in \parref{homogeneous-spinor}, the
ample generator \(\sO_X(1) \in \Pic X\) is the restriction of the Pl\"ucker
line bundle, so the task prescribed by
\parref{homogeneous-ampleness-criterion} of constructing sections of
\(\omega_X^{\otimes 1-p} \otimes \sO_X(-p)\) realizing
the elements \(\partial_{i,j} \tau\) amounts to producing suitable polynomials
in the minors of \(M\). In the cases considered, the level
decomposition from \parref{homogeneous-level} of the Lie algebra
\(\mathfrak{u}\) splits it into at most \(2\) indecomposable representations
for \(\mathbf{L}\), and so we need only explicitly construct a small number of
splittings by \parref{homogeneous-ampleness-u}.

\subsectiondash{Rolling determinants}\label{homogeneous-rolling-det}
In most of the cases considered, the device with which we construct the
requisite sections is the following: Let \(s > 0\) be the Fano index of \(X\),
so that \(\omega_X^\vee \cong \sO_X(s)\). Form an  auxiliary \(k \times (s + k
- 2)\) matrix consisting of columns from \(M\) of the form
\[
N =
\begin{pmatrix}
0 & \cdots & 0 & y_{1,1} & \cdots & y_{1,s-1} \\
1 & \cdots & 0 & y_{2,1} & \cdots & y_{2,s-1} \\
\vdots & \ddots & \vdots & \vdots & \ddots & \vdots \\
0 & \cdots & 1 & y_{k,1} & \cdots & y_{k,s-1}
\end{pmatrix}
\]
where the initial block consists of a size \(k-1\) identity matrix shifted down
one row, followed by other columns of \(M\) for some \(y_{i,j} = x_{i,j_i}\).
Construct a section of
\(\omega_X^{\otimes 1-p} \otimes \sO_X(1-p) \cong \sO_X((p-1)(s-1))\) by taking the
\((p-1)\)-power \emph{rolling determinant} of \(N\),
\[
(\rolldet N)^{p-1} \coloneqq
\prod\nolimits_{\ell = 1}^{s-1} (\det N_{\ell,\ldots,\ell+k-1})^{p-1}
\]
where \(N_{\ell,\ldots,\ell+k-1}\) is the \(k \times k\)-submatrix of
\(N\) formed by the columns from the \(\ell\)-th to the \((\ell+k-1)\)-st.
With a suitable choice of \(N\), this, or a minor variation thereof, often
provides a Frobenius splitting of \(X\) compatible with the Schubert divisor
\(X \setminus \mathbf{U}\) that is additionally divisible by
\((\det N_{1,\ldots,k})^{p-1} = y_{1,1}^{p-1}\). Dividing out \(y_{1,1}\) once
produces a section of \(\omega_X^{\otimes 1-p} \otimes \sO_X(-p)\) whose
image under the restriction map \(\res'\) to \(\kk[\mathbf{U}_1]_{< \tau}\)
contains \(\pm\partial\tau/\partial y_{1,1}\) and maybe other terms, which may be
determined by weight considerations. We then use the action of \(\mathfrak{u}\)
to show that the extraneous terms also lie in the image of \(\res'\), at
which point we may conclude.

The purpose of the rolling determinant is the following:

\begin{Lemma}\label{homogeneous-rolling-det-computation}
As an element of
\(\kk[y_{1,1},\ldots,y_{k,s-1}]/(y_{1,1}^p,\ldots,y_{k,s-1}^p)\),
\[
(\rolldet N)^{p-1} =
\Big(\prod\nolimits_{j = 1}^{s-k-1} \prod\nolimits_{i = 1}^k y_{i,j}\Big)^{p-1}
\Big(\prod\nolimits_{j = s-k}^{s-1} \prod\nolimits_{i = 1}^{s-j} y_{i,j}\Big)^{p-1}.
\]
\end{Lemma}

This can be established with an elementary inductive argument. Rather than
spell out the proof, we illustrate the key point by expanding the first three
factors in the rolling determinant. Thanks to the initial shifted identity
matrix block, this is the left side of
\[
y_{1,1}^{p-1}
\begin{vmatrix}
y_{1,1} & y_{1,2} \\ y_{2,1} & y_{2,2}
\end{vmatrix}^{p-1}
\begin{vmatrix}
y_{1,1} & y_{1,2} & y_{1,3} \\
y_{2,1} & y_{2,2} & y_{2,3} \\
y_{3,1} & y_{3,2} & y_{3,3}
\end{vmatrix}^{p-1}
=
y_{1,1}^{p-1}
\begin{vmatrix}
0 & y_{1,2} \\ y_{2,1} & y_{2,2}
\end{vmatrix}^{p-1}
\begin{vmatrix}
0 & y_{1,2} & y_{1,3} \\
y_{2,1} & y_{2,2} & y_{2,3} \\
y_{3,1} & y_{3,2} & y_{3,3}
\end{vmatrix}^{p-1}
\]
Since \(y_{i,j}^p = 0\), the initial \(y_{1,1}^{p-1}\) implies that subsequent
\(y_{1,1}\) do not contribute, and so we may set all later occurrences of
\(y_{1,1}\) to \(0\), as in the right side above. The second matrix is now
triangular, so its determinant is the product of its anti-diagonal terms
\(y_{2,1}\) and \(y_{1,2}\). As before, we may set these variables to
\(0\) in the third determinant, yielding again a triangular matrix, from which
it follows that the initial product of three determinants yields the element
\[
(y_{1,1} y_{2,1} y_{3,1} y_{1,2} y_{2,2}y_{1,3})^{p-1}
\in \kk[y_{1,1},\ldots,y_{k,s-1}]/(y_{1,1}^p,\ldots,y_{k,s-1}^p).
\]
Iterating this argument, each successive term
may be reduced to a determinant of a triangular matrix, from which the
lemma follows easily.

\subsectiondash{Grassmannian}\label{homogeneous-A}
As the first case of \parref{homogeneous-theorem}, consider \(X = \Gr(k,n)\) with
\(1 \leq k \leq n-1\). Take \(\mathbf{G} = \mathbf{GL}_n\), \(\mathbf{T}\)
the subgroup of diagonal matrices, \(\mathbf{P}\) the parabolic subgroup of
matrices with vanishing lower left \((n-k) \times k\) block, and Levi factor
\(\mathbf{L}\) the block diagonal subgroup
\(\mathbf{GL}_k \times \mathbf{GL}_{n-k}\). The action of \(\mathbf{L}\) on
\(\mathbf{U}\) may be described in terms of the matrix \(M\) from
\parref{homogeneous-strategy} as that induced by the simultaneous
action the left by \(\mathbf{GL}_k\) via row operations and on the right by
\(\mathbf{GL}_{n-k}\) via column operations on the initial \(n-k\) columns.
This implies that the adjoint representation on the Lie algebra of
\(\mathbf{U}\) is irreducible, isomorphic to the tensor product
\(\mathfrak{u} \cong \mathrm{Std}_{\mathbf{GL}_k} \otimes \mathrm{Std}_{\mathbf{GL}_{n-k}}\)
of the standard representations of \(\mathbf{GL}_k\) and \(\mathbf{GL}_{n-k}\).
By \parref{homogeneous-ampleness-u}, it suffices to construct the element
\(\partial_{1,1} \tau\).

Since \(\omega_X^{-1} \cong \sO_X(n)\) where \(\sO_X(1)\) is the Pl\"ucker,
the strategy from \parref{homogeneous-strategy} is to produce a
polynomial of degree \((p-1)(n-1)\) in the minors of \(M\) which provides a
Frobenius splitting divisible by the minor evaluating to \(x_{1,1}\).
As in \parref{homogeneous-rolling-det}, consider the
\(k \times (n+k-2)\) matrix formed from the columns of \(M\)
\[
A \coloneqq
\begin{pmatrix}
0 & \cdots & 0 & 0 & x_{1,1} & \cdots & x_{1,n-k} & 1 & 0 & \cdots & 0 \\
1 & \cdots & 0 & 0 & x_{2,1} & \cdots & x_{2,n-k} & 0 & 1 & \cdots & 0 \\
\vdots & \ddots & \vdots & \vdots & \vdots & \ddots & \vdots & \vdots & \vdots & \ddots & \vdots \\
0 & \cdots & 1 & 0 & x_{k-1,1} & \cdots & x_{k-1,n-k} & 0 & 0 & \cdots & 1 \\
0 & \cdots & 0 & 1 & x_{k, 1} & \cdots & x_{k,n-k} & 0 & 0 & \cdots & 0
\end{pmatrix}
\]
consisting of a size \(k-1\) identity matrix shifted down one,
followed by the \(k \times (n-k)\) block of coordinates of \(\mathbf{U}\),
followed by another size \(k-1\) identity, this time shifted up one. A slight
modification of \parref{homogeneous-rolling-det-computation} accounting for the
final identity matrix gives
\[
(\rolldet A)^{p-1} \coloneqq
\prod\nolimits_{\ell = 1}^{n-1} (\det A_{\ell,\ldots,\ell+k-1})^{p-1}
= \Big(\prod\nolimits_{j = 1}^{n - k}\prod\nolimits_{i = 1}^k  x_{i,j}\Big)^{p-1} \in
\kk[\mathbf{U}_1].
\]
Since the weight of \(x_{1,1}\) is indecomposable amongst the roots of
\(\mathbf{U}^+\), the weight space in \(\kk[\mathbf{U}_1]\) containing
\(\partial_{1,1}\tau\) is \(1\)-dimensional.  Dividing the displayed element by
\(\det A_{1,\ldots,k} = \pm x_{1,1}\) thus produces a section of
\(\omega_X^{\otimes 1-p} \otimes \sO_X(-p)\) restricting to
\(\pm\partial_{1,1}\tau\), completing the proof of \parref{homogeneous-theorem}
in type \(\mathrm{A}\).
\qed

\subsectiondash{Example}\label{homogeneous-A-Gr2}
Beyond the classical case of \(\PP^{n-1}\), the case \(X = \Gr(2,n)\) may be
alternatively seen using work of Raedschelders--\v{S}penko--Van den Bergh, in
which they explicitly compute the decomposition of \(F_*\sO_X\). Namely,
writing \(\mathcal{S}\) and \(\mathcal{Q}\) for the tautological sub- and
quotient bundles of ranks \(2\) and \(n-2\), respectively, they show in
\cite[Theorem 16.5 and Corollary 16.6]{RSSvB} that
\[
\mathcal{E}_X \cong
\Big(\bigoplus\nolimits_{k > 0} \sO_X(k)^{\oplus a_k}\Big) \oplus
\Big(\bigoplus\nolimits_{k > 0}\bigoplus\nolimits_{j = 1}^{n-3} \wedge^j \mathcal{Q}(k)^{\oplus b_{j,k}}\Big) \oplus
\Big(\bigoplus\nolimits_{k > 0}\bigoplus\nolimits_{j = 1}^{n-3} \mathcal{F}_j(j+k)^{\oplus c_{j,k}}\Big)
\]
for some multiplicities \(a_k\), \(b_{j,k}\), and \(c_{j,k}\). The
vector bundles \(\mathcal{F}_j\) are given by
\[
\begin{dcases*}
\mathcal{F}_j \cong \Sym^j(\mathcal{S}) & if \(0 \leq j \leq p-1\), \\
0 \to \Sym^{2p-2-j}(\mathcal{S}) \to \mathcal{F}_j \to \Sym^j(\mathcal{S}) \to 0 & if \(p \leq j \leq 2p-2\), and \\
\mathcal{F}_j \cong \mathcal{F}_{j_0} \otimes F^*\mathcal{F}_{j_1} \otimes \cdots \otimes F^{e,*}\mathcal{F}_{j_e} &
if \(j = j_0 + j_1p + \cdots + j_ep^e\)
\end{dcases*}
\]
where the extension is non-split in the second case, and
\(p-1 \leq j_i \leq 2p-2\) for \(i < e\) and \(0 \leq j_e \leq p-1\) in the
third case---note that our conventions are dual to those in \cite[\S4.3.1]{RSSvB}.
Standard arguments, using the wedge product isomorphism
\(\mathcal{S} \cong \mathcal{S}^\vee(-1)\) for instance, then show that all
terms appearing in \(\mathcal{E}_X(-1)\) are globally generated.
\qed

\subsectiondash{Symplectic Grassmannian}\label{homogeneous-C}
Next, consider \(X = \SGr(k,n)\) with \(n = 2m \geq 4\) and
\(2 \leq k \leq m\); we omit the case \(k = 1\) since then \(\SGr(1,n) \cong \PP^{n-1}\)
and this has been dealt with in \parref{homogeneous-A}.  Take
\(\mathbf{G} = \mathbf{Sp}_{2m}\) to be the symplectic group of matrices in
\(\mathbf{SL}_{2m}\) preserving the symplectic
form
\[
\omega \coloneqq
\sum\nolimits_{s = 1}^m \mathrm{d}x_s \wedge \mathrm{d}x_{2m+1-s}
= \mathrm{d}x_1 \wedge \mathrm{d}x_{2m} +
\mathrm{d}x_2 \wedge \mathrm{d}x_{2m-1} +
\cdots +
\mathrm{d}x_m \wedge \mathrm{d}x_{m+1}.
\]
With this choice, we may take
\(\mathbf{T}\) to be the subgroup of diagonal matrices, and \(\mathbf{P}\) the
parabolic subgroup preserving the two-step symplectic flag consisting of the
subspace spanned by the first \(k\) basis vectors together with its
\(\omega\)-orthogonal space. The variables \(x_{i,j}\) in the matrix \(M\) from
\parref{homogeneous-strategy} on the big Schubert cell of \(X\) satisfy the
following equations, which expresses isotropicity with respect to \(\omega\):
\[
x_{j,k+1-i} = 
x_{i,k+1-j} + \sum\nolimits_{s = k+1}^m (x_{i,s} x_{j,2m+1-s} - x_{i,2m+1-s} x_{j,s})
\;\;\text{for each \(1 \leq i < j \leq k\)}.
\]
We use these equations to eliminate the variables \(x_{j,k+1-i}\) strictly
below the anti-diagonal in \(M_{1,\ldots,k}\), so that the root variables of
\(\mathbf{U}\) are those on and above the anti-diagonal and also the variables
appearing in the central \(k \times (2m-2k)\) block.
To first order, \(M_{1,\ldots,k}\) is a matrix symmetric along its
anti-diagonal; for instance \(x_{k,k} = x_{1,1} + \mathfrak{m}^2\)
where \(\mathfrak{m}\) is the maximal ideal generated by all the variables.
Globally, \(X\) is cut out of \(\Gr(k,2m)\) by the regular section of the second
wedge power \(\wedge^2\mathcal{S}^\vee\) of the dual of the tautological
rank \(k\) subbundle, so
\[
\omega_X^{-1} \cong
\omega_{\Gr(k,2m)}^{-1}\otimes \det\wedge^2\mathcal{S}\rvert_X \cong
\sO_X(2m-k+1)
\]
where \(\sO_X(1)\) is the Pl\"ucker line bundle.

A Levi factor \(\mathbf{L}\) is isomorphic to
\(\mathbf{GL}_k \times \mathbf{Sp}_{2m-2k}\). Considering its action on \(\mathbf{U}\)
shows that the level decomposition of the adjoint representation on
\(\mathfrak{u}\) has two steps with
\[
\mathfrak{u} \cong
\mathrm{Std}_{\mathbf{GL}_k} \otimes \mathrm{Std}_{\mathbf{Sp}_{2m-2k}} \oplus
\Div^2\mathrm{Std}_{\mathbf{GL}_k}
\]
where level \(1\) is the tensor product of the standard representations
arising from the simultaneous left action of \(\mathbf{GL}_k\) by row operations
and right action of \(\mathbf{Sp}_{2m-2k}\) by column operations on the middle
\(k \times (2m-2k)\) block; and level \(2\) is the symmetric quadratic tensors
of the standard representation of \(\mathbf{GL}_k\) corresponding to its
action on the anti-diagonally symmetric \(M_{1,\ldots,k}\).
When \(p \geq 3\), both summands are irreducible representations for
\(\mathbf{L}\), and so by \parref{homogeneous-ampleness-u}, it suffices to
construct the elements \(\partial_{1,k+1}\tau\) and \(\partial_{1,1}\tau\)
as sections of \(\sO_X((p-1)(2m-k)-1)\).

For \(\partial_{1,k+1}\tau\), \parref{homogeneous-rolling-det-computation}
shows that the \((p-1)\)-power rolling determinant of the
\(k \times (2m-1)\) matrix
\[
C \coloneqq
\begin{pmatrix}
0 & \cdots & 0 &  x_{1,k+1} & \cdots & x_{1,2m-k} & x_{1,1} & \cdots & x_{1,k} \\
1 & \cdots & 0 &  x_{2,k+1} & \cdots & x_{2,2m-k} & x_{2,1} & \cdots & x_{2,k} \\
\vdots & \ddots & \vdots & \vdots & \ddots & \vdots & \vdots & \ddots & \vdots \\
0 & \cdots & 1 &  x_{k,k+1} & \cdots & x_{k,2m-k} & x_{k,1} & \cdots & x_{k,k}
\end{pmatrix}
\]
provides an expression for \(\tau\) divisible by
\(\det C_{1,\ldots,k} = \pm x_{1,k+1}\). Arguing via indecomposability of the
weight of \(x_{1,k+1}\) as in \parref{homogeneous-A} shows that the \((p-1)\)-power
rolling determinant of \(C\) divided by this factor produces a suitable
expression for \(\partial_{1,k+1}\tau\).

For \(\partial_{1,1}\tau\), let \(C'\) be the \(k \times (2m-1)\) matrix obtained from
\(C\) by moving the \(x_{i,1}\) column to the \(k\)-th position. Since the
variables \(x_{j,k+1-i}\) with \(1 \leq i < j \leq k\) that lie strictly below
the anti-diagonal of \(M_{1,\ldots,k}\) still do not appear in the product
given by applying \parref{homogeneous-rolling-det-computation} to \(C'\), the
\((p-1)\)-power rolling determinant of \(C'\) still provides an expression for
\(\tau\) divisible by \(\det C_{1,\ldots,k}' = \pm x_{1,1}\). However,
\(x_{1,1}\) has the same weight as \(x_{1,j}x_{k,2m+1-j}\)
for each \(k < j < 2m-k\), and so
\(\gamma' \coloneqq (\rolldet C')^{p-1}/\det C_{1,\ldots,k}'\) may contain
terms of the form \(\partial_{1,j}\partial_{k,2m+1-j}\tau\) in addition
to \(\partial_{1,1}\tau\). But \(\gamma'\), being a product of determinants, may
be seen to be a homogeneous polynomial of total degree \((p-1)\dim X - 1\) in
the matrix entries of \(C\). Thus every term in \(\gamma'\) has total degree
at least \((p-1)\dim X - 1\) in the root coordinates of \(\kk[\mathbf{U}_1]\).
This forces \(\gamma' = \pm\partial_{1,1}\tau\), completing the proof of
\parref{homogeneous-theorem} in type \(\mathrm{C}\).
\qed

\subsectiondash{Orthogonal Grassmannian}\label{homogeneous-BD}
Let \(X = \OGr(k,n)\) with \(n \geq 5\) and either
\(1 \leq k < \lfloor\frac{n-1}{2}\rfloor\) or
\(k = \lfloor\frac{n}{2}\rfloor\). Take \(\mathbf{G} = \mathbf{SO}_n\) to be
the subgroup of \(\mathbf{SL}_n\) preserving the quadratic form
\[
q \coloneqq
\begin{dcases*}
x_1x_{2m+1} + x_2x_{2m} + \cdots + x_m x_{m+2} + x_{m+1}^2 & if \(n = 2m+1\), and \\
x_1x_{2m} + x_2x_{2m-1} + \cdots + x_m x_{m+1} & if \(n = 2m\).
\end{dcases*}
\]
With this choice, we may take \(\mathbf{T}\) to be the subgroup of diagonal
matrices, and \(\mathbf{P}\) the parabolic subgroup preserving the two-step
orthogonal flag consisting of the subspace spanned by the first \(k\) basis vectors
together with its orthogonal space with respect to the polar form associated
to \(q\). The variables \(x_{i,j}\) of \(M\) from \parref{homogeneous-strategy}
on the big Schubert cell of \(X\) satisfy two sets of equations corresponding
to isotropicity with respect to \(q\). The first set expresses the
\(x_{i,k+1-i}\) for \(1 \leq i \leq k\) on the anti-diagonal of
\(M_{1,\ldots,k}\) as the quadratic form applied to the coordinates in the
middle \(k \times (n-2k)\) block:
\[
\begin{dcases*}
x_{i,k+1-i} = -x_{i,k+1} x_{i,n-k} - \cdots - x_{i,m} x_{i,m+2} - x_{i,m+1}^2 & if \(n = 2m+1\), and  \\
x_{i,k+1-i} = -x_{i,k+1} x_{i,n-k} - \cdots - x_{i,m} x_{i,m+1} & if \(n = 2m\).
\end{dcases*}
\]
The second set of equations shows that \(M_{1,\ldots,k}\) is, to first-order,
anti-symmetric along the anti-diagonal: that is, for \(1 \leq i < j \leq k\),
\[
\begin{dcases*}
x_{j,k+1-i} = -x_{i,k+1-j} 
-x_{i,k+1} x_{j,n-k} - \cdots - x_{i,m} x_{j,m+2} - 2x_{i,m+1}x_{j,m+1} & if \(n = 2m+1\), and \\
x_{j,k+1-i} = -x_{i,k+1-j} -
x_{i,k+1} x_{j,n-k} - \cdots - x_{i,m} x_{j,m+1} & if \(n = 2m\).
\end{dcases*}
\]
We use these equations to eliminate the variables on and below the anti-diagonal
of \(M_{1,\ldots,k}\), so that the root coordinates are given by those above
the anti-diagonal of \(M_{1,\ldots,k}\) and the variables in the central
\(k \times (n-2k)\) block of \(M\). Globally, \(X\) is cut out of \(\Gr(k,n)\)
by the regular section induced by \(q\) of the second symmetric power
\(\Sym^2\mathcal{S}^\vee\) of the tautological rank \(k\) subbundle, so that
\[
\omega_X^{-1}
\cong \omega_{\Gr(k,n)}^{-1} \otimes \det \Sym^2\mathcal{S} \rvert_X
\cong \mathcal{L}^{\otimes n-k-1}
\]
where \(\mathcal{L}\) is the restriction of the Pl\"ucker line bundle onto \(X\).
If \(k < m\), then \(\mathcal{L} = \sO_X(1)\) is the ample generator of \(\Pic X\)
and otherwise \(\mathcal{L} = \sO_X(2)\) in the case \(k = m\): see
\parref{homogeneous-spinor} for details.

\subsectiondash{}\label{homogeneous-BD-submaximal}
Consider first the case where \(1 \leq k < \lfloor\frac{n-1}{2}\rfloor\); in
terms of Dynkin diagrams, this means that the parabolic \(\mathbf{P}\) is not
one of the spinor nodes on the right. In this case, \(\sO_X(1)\) is the
Pl\"ucker line bundle and a Levi factor \(\mathbf{L}\) is isomorphic to
\(\mathbf{GL}_k \times \mathbf{SO}_{n-2k}\). Computations analogous to the
symplectic case shows that the level decomposition of the Lie algebra of
\(\mathbf{U}\) is
\[
\mathfrak{u} \cong
\mathrm{Std}_{\mathbf{GL}_k} \otimes \mathrm{Std}_{\mathbf{SO}_{n-2k}} \oplus
\wedge^2 \mathrm{Std}_{\mathbf{GL}_k}
\]
where level \(1\) is the tensor product of the standard representations arising
as the action on the central \(k \times (n-2k)\) block of \(M\)
via row operations by \(\mathbf{GL}_k\) and column operations by
\(\mathbf{SO}_{n-2k}\); and level \(2\) comes from the
action of \(\mathbf{GL}_k\) on the anti-diagonally anti-symmetric matrix
\(M_{1,\ldots,k}\). The wedge square is always irreducible, whereas the
tensor product of standard representations is irreducible except in the case
the characteristic \(p = 2\) and \(n = 2m+1\) is odd; we treat this exceptional
case in \parref{homogeneous-BD-exceptional}. In the remaining cases, it
suffices to construct sections of \(\sO_X((p-1)(n-k-2)-1)\) realizing the
elements \(\partial_{1,k+2}\tau\) and \(\partial_{1,1}\tau\). We proceed in three steps:

\textbf{Step 1.}
To construct \(\partial_{1,k+2}\tau\), consider the \(k \times (n-3)\) matrix
\[
B \coloneqq
\begin{pmatrix}
0 & \cdots & 0 & x_{1,k+2} & \cdots &  x_{1,n-k-1} & x_{1,1} & \cdots & x_{1,k} \\
1 & \cdots & 0 & x_{2,k+2} & \cdots &  x_{2,n-k-1} & x_{2,1} & \cdots & x_{2,k} \\
\vdots & \ddots & \vdots & \vdots & \ddots &  \vdots & \vdots & \ddots & \vdots \\
0 & \cdots & 1 & x_{k,k+2} & \cdots &  x_{k,n-k-1} & x_{k,1} & \cdots & x_{k,k}
\end{pmatrix}
\]
where the pair of columns with the variables \(x_{i,k+1}\) and
\(x_{i,n-k}\) do not appear; the essential property of the missing columns is
that they are matched by the quadratic form \(q\). Applying
\parref{homogeneous-rolling-det-computation} gives
\[
(\rolldet B)^{p-1} =
\Big(\prod\nolimits_{j = k+2}^{n-k-1}\prod\nolimits_{i = 1}^k x_{i,j}\Big)^{p-1}
\Big(\prod\nolimits_{j = 1}^k \prod\nolimits_{i = 1}^{k+1-j}  x_{i,j}\Big)^{p-1}
\in \kk[\mathbf{U}_1]
\]
which is the \((p-1)\)-power of the product of all the variables in \(M\)
excluding the \(x_{i,k+1}\), \(x_{i,n-k}\), and those strictly below
the anti-diagonal in \(M_{1,\ldots,k}\). Consider the product of the
anti-diagonal matrix entries \(x_{i,k+1-i}\) in the second product: the
first set of equations of \(X\) in \parref{homogeneous-BD} reads
\[
x_{i,k+1-i} = 
-x_{i,k+1}x_{i,n-k} - q'_i
\;\;\text{for each \(i = 1,\ldots,k\)}
\]
where \(q'_i\) is a quadratic form in the variables \(x_{i,j}\) with \(k+2 \leq
j \leq n-k-1\). Since each variable in \(q'_i\) already appears in the rolling
determinant with exponent \(p-1\), each \(x_{i,k+1-i}\) only contributes
\(x_{i,k+1}x_{i,n-k}\) into the product. Therefore
\begin{multline*}
\Big(\prod\nolimits_{j = k+2}^{n-k-1}\prod\nolimits_{i = 1}^k  x_{i,j}\Big)^{p-1}
\Big(\prod\nolimits_{j = 1}^{k - 1}\prod\nolimits_{i = 1}^{k-j}  x_{i,j}\Big)^{p-1}
\Big(\prod\nolimits_{i = 1}^k x_{i,k+1-i}\Big)^{p-1} \\ =
\Big(\prod\nolimits_{j = k+2}^{n-k-1}\prod\nolimits_{i = 1}^k  x_{i,j}\Big)^{p-1}
\Big(\prod\nolimits_{j = 1}^{k - 1}\prod\nolimits_{i = 1}^{k-j}  x_{i,j}\Big)^{p-1}
\Big(\prod\nolimits_{i = 1}^k x_{i,k+1}x_{i,n-k}\Big)^{p-1}
\end{multline*}
and this is precisely the trace element \(\tau\) in \(\kk[\mathbf{U}_1]\).
Since the weight of \(x_{1,k+2}\) is indecomposable in \(\Phi^+_{\mathbf{P}}\),
dividing this expression for \(\tau\) by \(\det B_{1,\ldots,k} = \pm x_{1,k+2}\)
provides a suitable expression for \(\partial_{1,k+2}\tau\). Irreducibility
of the \(\mathbf{L}\)-representation
\[
\kk[\mathbf{U}_1]_{N-1} \cong
\mathfrak{u}_1 \cdot \tau \cong
\mathrm{Std}_{\mathbf{GL}_k} \otimes \mathrm{Std}_{\mathbf{SO}_{n-2k}}
\]
now implies that each of the elements \(\partial_{i,j}\tau\) for
\(1 \leq i \leq k\) and \(k+2 \leq j \leq n-k-1\) may be realized via a
section of \(\omega_X^{\otimes 1-p} \otimes \sO_X(-p)\).

\textbf{Step 2.}
We now show that the elements \(\partial_{1,j}\partial_{k,n+1-j}\tau\) with
\(k + 1 \leq j \leq n-k\) may also be realized by sections of
\(\omega_X^{\otimes 1-p} \otimes \sO_X(-p)\). As discussed in
\parref{homogeneous-ampleness-lie-algebra}, the restriction map \(\res'\) is
equivariant for the action of the Lie algebra \(\mathfrak{u}\), so it suffices
to show that the \(\partial_{1,j}\partial_{k,n+1-j}\tau\) may be obtained from
the \(\partial_{i,j}\tau\) constructed in step 1 via the action of \(\mathfrak{u}\).
For each \(k+1 \leq j \leq n-k\), consider the Lie algebra elements \(X_{1,j}\)
and \(X_{k,n+1-j}\) dual to the root coordinates \(x_{1,j}\) and
\(x_{k,n+1-j}\), respectively.  Weight considerations show that
they act on \(\kk[\mathbf{U}_1]\) as derivations of the form
\[
X_{1,j} = \partial_{1,j} + c x_{k,n+1-j}\partial_{1,1} + D
\;\;\text{and}\;\;
X_{k,n+1-j} = \partial_{k,n+1-j} + c' x_{1,j} \partial_{1,1} + D'
\]
for scalars \(c,c' \in \kk\), and derivations \(D\) and \(D'\)
which do not involve any of \(\partial_{1,j}\),
\(\partial_{k,n+1-j}\), and \(\partial_{1,1}\). Compute the commutator
\([X_{1,j}, X_{k,n+1-j}]\) in two ways: On the one hand, a direct
computation with the derivations gives the first equality in
\[
(c'-c) X_{1,1} =
[X_{1,j},X_{k,n+1-j}] =
c'' X_{1,1}
\]
where \(X_{1,1} = \partial_{1,1}\) is the derivation dual to the root
subgroup with coordinate \(x_{1,1}\). On the other hand, \(\mathfrak{u}\)
is a subalgebra of \(\mathfrak{so}_n\), so the commutator must be of the
form \(c''X_{1,1}\) where the coefficient is a multiple of Chevalley's
structure constant indexed by the weights of \(x_{1,j}\) and \(x_{k,n+1-j}\).
Consulting \cite[Ch. VIII \S2.4, Prop. 7]{Bourbaki:Lie-7-9}, for example,
shows that all structure constants are \(\pm 1\) except in the case
\(n = 2m+1\) and \(j = n-m = m+1\), wherein the structure constants between
\(x_{1,m+1}\) and \(x_{k,m+1}\) is \(\pm 2\). Since we treat the odd \(n\)
case in characteristic \(p = 2\) separately in
\parref{homogeneous-BD-exceptional}, we have \(c'' \neq 0\). A
direct calculation now shows that
\[
\partial_{1,j}\partial_{k,n+1-j}\tau =
\begin{dcases*}
X_{1,j} \cdot \partial_{k,n+1-j}\tau & if \(c = 0\), \\
X_{k,n+1-j} \cdot \partial_{1,j}\tau & if \(c' = 0\), and \\
(c'X_{1,j} \cdot \partial_{k,n+1,j}\tau - cX_{k,n+1-j} \cdot \partial_{1,j}\tau)/c'' & if \(cc' \neq 0\).
\end{dcases*}
\]
In all cases, this shows that \(\partial_{i,j}\partial_{k,n+1-j}\tau\) is
in the image of \(\res'\).

\textbf{Step 3.}
Finally, to construct \(\partial_{1,1}\tau\), let \(B'\) be the
\(k \times (n-3)\) matrix obtained from \(B\) by moving the \(x_{i,1}\) column
to the \(k\)-th position. Computing as in step 1 together with weight
considerations give
\[
\frac{(\rolldet B')^{p-1}}{\det B'_{1,\ldots,k}} =
\pm\partial_{1,1}\tau +
\sum_{j = k+1}^{n-k} c_j \cdot \partial_{1,j}\partial_{k,n+1-j}\tau.
\]
for some scalars \(c_1,\ldots,c_{n-k} \in \kk\) which are typically not
all zero. Step 2 shows that each term in the second summand already
lies in the image of \(\res'\), whence this implies \(\partial_{1,1}\tau\)
does too.

\subsectiondash{}\label{homogeneous-BD-exceptional}
To complete the proof of \parref{homogeneous-theorem} for orthogonal
Grassmannian other than the spinor varieties, it remains to consider the odd
orthogonal Grassmannian \(X = \OGr(k,2m+1)\) with \(1 \leq k \leq m-2\) in
characteristic \(p = 2\). Each of the steps in
\parref{homogeneous-BD-submaximal} requires a slight modification:

\textbf{Step 1\('\).}
The standard representation \(\mathrm{Std}_{\mathbf{SO}_{2m+1-2k}}\) of the
odd orthogonal groups is not simple: \(\mathbf{SO}_{2m+1-2k}\) preserves
the \(1\)-dimensional kernel of the polar form \(b_q\) associated with the
quadratic form \(q\); see \parref{exceptions-symplectic} for details. Thus
there is a short exact sequence
\[
0 \to
\mathrm{Std}_{\mathbf{GL}_k} \otimes \operatorname{rad} b_q \to
\mathrm{Std}_{\mathbf{GL}_k} \otimes \mathrm{Std}_{\mathbf{SO}_{2m+1-2k}} \to
\mathrm{Std}_{\mathbf{GL}_k} \otimes\mathrm{Std}_{\mathbf{SO}_{2m+1-2k}}/\operatorname{rad} b_q \to
0
\]
where the bilinear radical on the left is the \(1\)-dimensional trivial
representation for \(\mathbf{SO}_{2m+1-2k}\) spanned by the vector dual to
the coordinate \(x_{m+1}\). This implies that when \(k \neq m - 1\), the
element \(\partial_{1,k+2}\tau\) realized in step 1 of
\parref{homogeneous-BD-submaximal} lies in the quotient above; a
section of \(\omega_X^{\otimes 1-p} \otimes \sO_X(-p)\) realizing
\(\partial_{1,m+1}\tau\) in the subrepresentation may then be constructed as
usual from the \((p-1)\)-power rolling determinant of the matrix obtained from
\(B\) by moving the \(x_{i,m+1}\) column to the \(k\)-th position.
Since the sub- and quotient representations are individually simple for
\(\mathbf{L}\), this produces each of the elements \(\partial_{i,j}\tau\) for
\(1 \leq i \leq k\) and \(k+1 \leq j \leq n-k\).

\textbf{Step 2\('\).}
The argument in step 2 of \parref{homogeneous-BD-submaximal} still shows that
the elements \(\partial_{1,j}\partial_{k,n+1-j}\tau\) lie in the image
of \(\res'\) for \(k+1 \leq j \leq n-k\), except maybe for
\(\partial_{1,m+1}\partial_{k,m+1}\tau\). In fact, a more
careful calculation shows that, although \(c'' = 0\) in this last
case, \(cc' = 0\) and so the argument goes through; rather than
explaining this, we instead do a more careful calculation in the next step.

\textbf{Step 3\('\).}
To obtain \(\partial_{1,1}\tau\) in this case, the divided rolling determinant
nonetheless gives
\[
\beta' \coloneqq
\prod\nolimits_{\ell = 2}^{2m-k-1} \det B_{\ell,\ldots,\ell+k-1}' =
\partial_{1,1}\tau +
\sum\nolimits_{j = k+1}^{2m+1-k} c_j \cdot \partial_{1,j}\partial_{k,2m+2-j}\tau.
\]
It suffices to show that \(c_{m+1} = 0\), at which point the argument
of step 3 from \parref{homogeneous-BD-submaximal} yields the result. Applying
\parref{homogeneous-rolling-det-computation} shows that \(\beta'\) has total
degree
\[
\deg \beta'
= k(n-2k-2) + k(k+1)/2 - 1
= \dim X - k - 1
\]
in the entries of \(B'\). Since \(B'\) contains only
\(\dim X - 2k\) root coordinates of \(X\), every term in \(\beta'\)
contains at least \(k-1\) of the \(x_{j,k+1-i}\) with
\(1 \leq i \leq j \leq k\) on or below the anti-diagonal of \(M_{1,\ldots,k}\).
The equations of \parref{homogeneous-BD} show that the \(x_{j,k+1-i}\) are
quadratic in the root coordinates, so degree considerations imply that each
term of \(\beta'\) contains either \(k\) or \(k-1\) contributions
from these quadratic terms; moreover, in the former case,
all nonzero terms are proportional to \(\partial_{1,1}\tau\). In other words,
each of degree \(\dim X - 2\) terms in \(\beta'\) are obtained from
\(\tau\) by omitting a quadratic term appearing in the lower anti-triangular
variables \(x_{j,k+1-i}\).  Examining the equations in \parref{homogeneous-BD}
shows that the variables in the \(m+1\) column either comes as a square
\(x_{i,m+1}^2\) or else with an even coefficient as \(2x_{i,m+1}x_{j,m+1}\); in
both cases, these vanish in \(\kk[\mathbf{U}_1]\). Thus the \(x_{i,m+1}\)
cannot be omitted in the rolling determinant above, so \(c_{m+1} = 0\). This
completes the proof of \parref{homogeneous-theorem} for \(X = \OGr(k,n)\) with
\(1 \leq k < \lfloor\frac{n}{2}\rfloor\).

\subsectiondash{Spinor variety}\label{homogeneous-spinor}
Consider the orthogonal Grassmannian corresponding to a maximal parabolic
indexed by one of the right-most vertices in the Dynkin diagrams
\[
\mathrm{B}_{m-1} \colon {\dynkin B{}}
\;\;\;\;\text{and}\;\;\;\;
\mathrm{D}_m \colon {\dynkin D{}}
\]
these are the \emph{spinor varieties}, all of which are isomorphic:
\[
X \coloneqq \OGr(m-1,2m-1) = \OGr_+(m,2m) = \OGr_-(m,2m).
\]
Geometrically, we view the spinor variety as parameterizing totally isotropic
subspaces of dimension \(m\) for a quadratic form \(q\) of dimension
\(n = 2m\). The matrix \(M\) of local coordinates from
\parref{homogeneous-strategy} on the Schubert cell provided by \(\mathbf{U}\)
is therefore of size \(m \times 2m\), and the equations in
\parref{homogeneous-BD} simply say that the initial \(M_{1,\ldots,m}\) is
anti-symmetric along its anti-diagonal; in what follows, it is slightly
more convenient to reorder the coordinates and work with the
anti-symmetric conjugate
\[
D \coloneqq
\begin{pmatrix}
0 & x_{1,m-1} & \cdots & x_{1,2} & x_{1,1} \\
 & 0 & \cdots &  x_{2,2} & x_{2,1} \\
 & & \ddots & \vdots & \vdots \\
 &  & & 0 & x_{m-1,1} \\
 &  & &  & 0
\end{pmatrix}
\eqqcolon
\begin{pmatrix}
0 & y_{1,2} & \cdots & y_{1,m-1} & y_{1,m} \\
  & 0 & \cdots &  y_{2,m-1} & y_{2,m} \\
& & \ddots & \vdots & \vdots \\
 &  & & 0 & y_{m-1,m} \\
 &  & &  & 0
\end{pmatrix}
\]
where we display only the upper triangular half; for the computation below,
it will also be convenient to rename the variables \(x_{i,m+1-j}\) to the
\(y_{i,j}\) which are indexed with respect to their usual position in \(D\).
A Levi factor is \(\mathbf{L} = \mathbf{GL}_m\) and
its adjoint representation on \(\mathfrak{u} \cong \wedge^2
\mathrm{Std}_{\mathbf{GL}_m}\) is irreducible, being the alternating square of
the standard representation arising from its action on alternating matrices.

Global sections of the ample generator \(\sO_X(1) \in \Pic X\) are now spanned
by \emph{Pfaffians} of principal anti-symmetric submatrices of \(D\); see
\cite[Appendix D]{FP:Degeneracy} for more on Pfaffians. That the restriction of
the Pl\"ucker line bundle satisfies \(\mathcal{L} \cong \sO_X(2)\) reflects the
identity \(\det D = (\pf D)^2\). The canonical bundle computation of
\parref{homogeneous-BD} means here that we are to construct the trace element
\(\tau\) as a polynomial of degree \((p-1)(2m-3)\) in the Pfaffians of
anti-symmetric submatrices of \(D\).

To describe the construction, use the following notation: For integers \(1 \leq
i \leq j \leq m\), write \([i,j]\) for the set of integers between \(i\) and
\(j\); for a subset \(I \subseteq [1,m]\), let \(D_I\) denote the principal
anti-symmetric submatrix of \(D\) consisting of those rows and columns indexed
by \(I\). Define the \emph{rolling Pfaffian} of \(D\) depending on the parity
of \(m\). If \(m = 2k\) is even, let
\[
\rollpf D \coloneqq
\Big(\prod\nolimits_{s = 1}^{k-1} \pf D_{[1,2s]} \cdot \pf D_{[1,2s+1] \setminus (s+1)}\Big)
\cdot
\pf D
\cdot
\Big(\prod\nolimits_{s = 1}^{k-1} \pf D_{[2s,2k]\setminus(k+s)} \cdot \pf D_{[2s+1,2k]}\Big).
\]
In the case \(m = 2k+1\) is odd, define
\begin{multline*}
\rollpf D \coloneqq
\Big(\prod\nolimits_{s = 1}^{k-1} \pf D_{[1,2s]} \cdot \pf D_{[1,2s+1] \setminus (s+1)}\Big)
\cdot \pf D_{[1,2k]}  \\ \cdot \pf D_{[1,2k+1] \setminus (k+1)} \cdot \pf D_{[2,2k+1]} \cdot
\Big(\prod\nolimits_{s = 1}^{k-1} \pf D_{[2s,2k+1]} \cdot \pf D_{[2s+1,2k+1] \setminus (k+s+1)}\Big).
\end{multline*}
The matrix appearing in the \(t\)-th term of the rolling Pfaffian is the
smallest principal submatrix of \(D\) containing the anti-diagonal starting
from \(y_{1,t+1}\) for \(1 \leq t \leq m-1\) and \(y_{t + 2 - m, m}\) for
\(m \leq t \leq 2m-3\); see Figure \parref{homogeneous-spinor-pfaffian.figure}.
The \((p-1)\)-power rolling Pfaffian of \(D\) realizes the trace element:

\begin{figure}[t!]
\centering
\begin{subfigure}[t]{0.4\textwidth}
\begin{tikzpicture}[
  every left delimiter/.style={xshift=0.7em},
  every right delimiter/.style={xshift=-0.7em}]
\matrix(m)[matrix of math nodes,
           left delimiter=(,
           right delimiter=),
           column sep=0.5em,
           row sep=0.25em]{
0 & y_{1,2} & y_{1,3} & y_{1,4} & y_{1,5} & y_{1,6} \\
  & 0       & y_{2,3} & y_{2,4} & y_{2,5} & y_{2,6} \\
  &         & 0       & y_{3,4} & y_{3,5} & y_{3,6} \\
  &         &         & 0       & y_{4,5} & y_{4,6} \\
  &         &         &         & 0       & y_{5,6} \\
  &         &         &         &         & 0 \\
};

\draw[rounded corners=5pt, thick]
([yshift=-2pt]m-1-2.east) |-
([xshift=-2pt]m-1-2.north) --
([yshift=2pt]m-1-2.west) |-
([xshift=2pt]m-1-2.south) --
cycle;
\draw[rounded corners=5pt, thick]
([yshift=-2pt]m-1-3.east) |-
([xshift=-2pt]m-1-3.north) --
([yshift=2pt]m-1-3.west) |-
([xshift=2pt]m-1-3.south) --
cycle;
\draw[rounded corners=5pt, thick]
([yshift=-2pt]m-1-4.east) |-
([xshift=-2pt]m-1-4.north) --
([yshift=2pt]m-2-3.west) |-
([xshift=2pt]m-2-3.south) --
cycle;
\draw[rounded corners=5pt, thick]
([yshift=-2pt]m-1-5.east) |-
([xshift=-2pt]m-1-5.north) --
([yshift=2pt]m-2-4.west) |-
([xshift=2pt]m-2-4.south) --
cycle;
\draw[rounded corners=5pt, thick]
([yshift=-2pt]m-1-6.east) |-
([xshift=-2pt]m-1-6.north) --
([yshift=2pt]m-3-4.west) |-
([xshift=2pt]m-3-4.south) --
cycle;
\draw[rounded corners=5pt, thick]
([yshift=-2pt]m-2-6.east) |-
([xshift=-2pt]m-2-6.north) --
([yshift=2pt]m-3-5.west) |-
([xshift=2pt]m-3-5.south) --
cycle;
\draw[rounded corners=5pt, thick]
([yshift=-2pt]m-3-6.east) |-
([xshift=-2pt]m-3-6.north) --
([yshift=2pt]m-4-5.west) |-
([xshift=2pt]m-4-5.south) --
cycle;
\draw[rounded corners=5pt, thick]
([yshift=-2pt]m-4-6.east) |-
([xshift=-2pt]m-4-6.north) --
([yshift=2pt]m-4-6.west) |-
([xshift=2pt]m-4-6.south) --
cycle;
\draw[rounded corners=5pt, thick]
([yshift=-2pt]m-5-6.east) |-
([xshift=-2pt]m-5-6.north) --
([yshift=2pt]m-5-6.west) |-
([xshift=2pt]m-5-6.south) --
cycle;
\end{tikzpicture}
\end{subfigure}%
~
\begin{subfigure}[t]{0.5\textwidth}
\begin{tikzpicture}[
  every left delimiter/.style={xshift=0.7em},
  every right delimiter/.style={xshift=-0.7em}]
\matrix(m)[matrix of math nodes, left delimiter=(, right delimiter=), column sep=0.5em, row sep=0.2em]{
0 & y_{1,2} & y_{1,3} & y_{1,4} & y_{1,5} & y_{1,6} & y_{1,7} \\
  & 0       & y_{2,3} & y_{2,4} & y_{2,5} & y_{2,6} & y_{2,7} \\
  &         & 0       & y_{3,4} & y_{3,5} & y_{3,6} & y_{3,7} \\
  &         &         & 0       & y_{4,5} & y_{4,6} & y_{4,7} \\
  &         &         &         & 0       & y_{5,6} & y_{5,7} \\
  &         &         &         &         & 0       & y_{6,7} \\
  &         &         &         &         &         & 0 \\
};
\draw[rounded corners=5pt, thick]
([yshift=-2pt]m-1-2.east) |-
([xshift=-2pt]m-1-2.north) --
([yshift=2pt]m-1-2.west) |-
([xshift=2pt]m-1-2.south) --
cycle;
\draw[rounded corners=5pt, thick]
([yshift=-2pt]m-1-3.east) |-
([xshift=-2pt]m-1-3.north) --
([yshift=2pt]m-1-3.west) |-
([xshift=2pt]m-1-3.south) --
cycle;
\draw[rounded corners=5pt, thick]
([yshift=-2pt]m-1-4.east) |-
([xshift=-2pt]m-1-4.north) --
([yshift=2pt]m-2-3.west) |-
([xshift=2pt]m-2-3.south) --
cycle;
\draw[rounded corners=5pt, thick]
([yshift=-2pt]m-1-5.east) |-
([xshift=-2pt]m-1-5.north) --
([yshift=2pt]m-2-4.west) |-
([xshift=2pt]m-2-4.south) --
cycle;
\draw[rounded corners=5pt, thick]
([yshift=-2pt]m-1-6.east) |-
([xshift=-2pt]m-1-6.north) --
([yshift=2pt]m-3-4.west) |-
([xshift=2pt]m-3-4.south) --
cycle;
\draw[rounded corners=5pt, thick]
([yshift=-2pt]m-1-7.east) |-
([xshift=-2pt]m-1-7.north) --
([yshift=2pt]m-3-5.west) |-
([xshift=2pt]m-3-5.south) --
cycle;
\draw[rounded corners=5pt, thick]
([yshift=-2pt]m-2-7.east) |-
([xshift=-2pt]m-2-7.north) --
([yshift=2pt]m-4-5.west) |-
([xshift=2pt]m-4-5.south) --
cycle;
\draw[rounded corners=5pt, thick]
([yshift=-2pt]m-3-7.east) |-
([xshift=-2pt]m-3-7.north) --
([yshift=2pt]m-4-6.west) |-
([xshift=2pt]m-4-6.south) --
cycle;
\draw[rounded corners=5pt, thick]
([yshift=-2pt]m-4-7.east) |-
([xshift=-2pt]m-4-7.north) --
([yshift=2pt]m-5-6.west) |-
([xshift=2pt]m-5-6.south) --
cycle;
\draw[rounded corners=5pt, thick]
([yshift=-2pt]m-5-7.east) |-
([xshift=-2pt]m-5-7.north) --
([yshift=2pt]m-5-7.west) |-
([xshift=2pt]m-5-7.south) --
cycle;
\draw[rounded corners=5pt, thick]
([yshift=-2pt]m-6-7.east) |-
([xshift=-2pt]m-6-7.north) --
([yshift=2pt]m-6-7.west) |-
([xshift=2pt]m-6-7.south) --
cycle;
\end{tikzpicture}
\end{subfigure}
\caption{The matrices appearing in the rolling Pfaffian are the smallest
anti-symmetric principal submatrices that contain the anti-diagonals starting
from the top and right border of the matrix \(D\), as is illustrated in the
cases \(m = 6\) and \(m = 7\).}
\label{homogeneous-spinor-pfaffian.figure}
\end{figure}
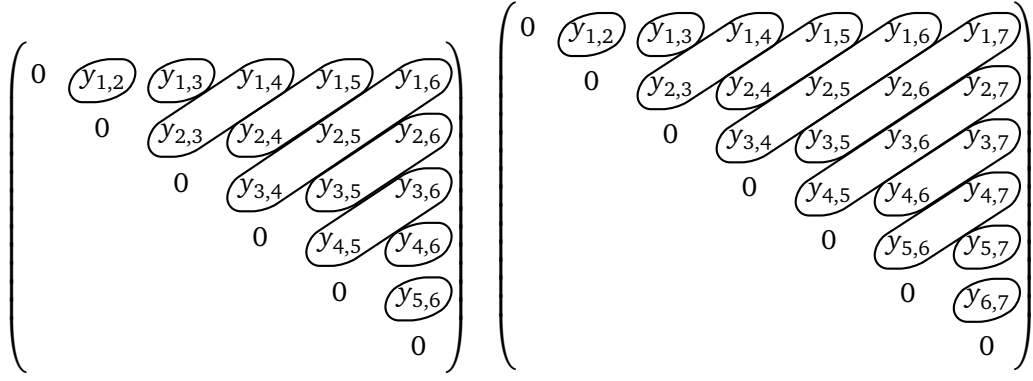

\begin{Lemma}\label{homogeneous-spinor-pfaffian}
\(\displaystyle
(\rollpf D)^{p-1} =
\Big(\prod\nolimits_{i = 1}^{m-1} \prod\nolimits_{j = 1}^{m-i} x_{i,j}\Big)^{p-1}
\in \kk[\mathbf{U}_1]\).
\end{Lemma}

This may be established with an elementary inductive argument.
As with \parref{homogeneous-rolling-det-computation}, the main observation
being that each successive term in the product is the Pfaffian of an
anti-symmetric matrix which is triangular along its lower anti-diagonal. For
example, if \(m \geq 5\), the first four terms appearing in the \((p-1)\)-power
rolling Pfaffian give, in the ring
\(\kk[y_{i,j} : 1 \leq i < j \leq m]/(y_{i,j}^p : 1 \leq i < j \leq m)\),
\begin{multline*}
\begin{vmatrix} 0 & y_{1,2} \\ & 0\end{vmatrix}^{p-1}
\begin{vmatrix} 0 & y_{1,3} \\ & 0\end{vmatrix}^{p-1}
\begin{vmatrix} 0 & y_{1,2} & y_{1,3} & y_{1,4} \\ & 0 & y_{2,3} & y_{2,4} \\ & & 0 & y_{3,4} \\ & & & 0 \end{vmatrix}^{p-1}
\begin{vmatrix} 0 & y_{1,2} & y_{1,4} & y_{1,5} \\ & 0 & y_{2,4} & y_{2,5} \\ & & 0 & y_{4,5} \\ & & & 0 \end{vmatrix}^{p-1} \\
= (y_{1,2}y_{1,3})^{p-1}
\begin{vmatrix} 0 & 0 & 0 & y_{1,4} \\ & 0 & y_{2,3} & y_{2,4} \\ & & 0 & y_{3,4} \\ & & & 0 \end{vmatrix}^{p-1}
\begin{vmatrix} 0 & 0 & y_{1,4} & y_{1,5} \\ & 0 & y_{2,4} & y_{2,5} \\ & & 0 & y_{4,5} \\ & & & 0 \end{vmatrix}^{p-1}
= (y_{1,2}y_{1,3}y_{1,4}y_{2,3}y_{1,5}y_{2,4})^{p-1}
\end{multline*}
where the vertical bars here denote the Pfaffian of the corresponding
anti-symmetric matrix. We leave the proof of \parref{homogeneous-spinor-pfaffian}
to the reader.

To conclude, observe that \((\rollpf D)^{p-1}\) defines a section of
\(\sO_X((p-1)(2m-3))\) and is divisible by the Pfaffian \(x_{1,m-2}\) of the
\(2 \times 2\) anti-symmetric submatrix \(D_{[1,2]}\). Since the weight of this
variable is indecomposable in \(\Phi^+_{\mathbf{P}}\), the expression
\((\rollpf D)^{p-1}/\pf D_{[1,2]}\) provides a section of
\(\omega_X^{\otimes 1-p} \otimes \sO_X(-p) \cong \sO_X((p-1)(2m-3)-1)\) realizing
\(\partial_{1,m-1}\tau\) in \(\kk[\mathbf{U}_1]\). Since \(\mathfrak{u} \cong
\wedge^2\mathrm{Std}_{\mathbf{GL}_m}\) is an irreducible representation for
\(\mathbf{L} \cong \mathbf{GL}_m\), all the sections \(\partial_{i,j}\tau\) are
in the image of the restriction map, and so the conclusion follows from
\parref{homogeneous-ampleness-criterion}.  This completes the proof of
\parref{homogeneous-theorem} in types \(\mathrm{B}\) and \(\mathrm{D}\).
\qed

\subsectiondash{Example}\label{homogeneous-quadrics}
When \(k = 1\), the orthogonal Grassmannian \(X = \OGr(1,n)\) is but a smooth
quadric hypersurface in \(\PP^{n-1}\). Thus \parref{homogeneous-theorem} says
that \(\mathcal{E}_X(-1)\) is globally generated for smooth quadric hypersurfaces
with \(\dim X \geq 3\), except maybe in the case \(\dim X = 3\) and \(p = 2\).
And indeed, \(\mathcal{E}_X(-1)\) is not globally generated in the latter case:
This follows, for instance, from the discussion of
\parref{homogeneous-quadric-threefold}. More directly, observe that the
restriction map of \parref{homogeneous-ampleness-computation} in this case is a
linear map of the form
\[
\res' \colon
\mathrm{H}^0(X,\sO_X(1)) \to
\kk[\mathbf{U}_1]_{< \tau}.
\]
This cannot be surjective since, the source has dimension \(5\) whereas the
target has dimension \(7\).



\subsectiondash{Type \(\mathbf{G}_2\)}\label{homogeneous-G}
Let \(\mathbf{G}\) be a semisimple algebraic group with Dynkin diagram
\(\dynkin G2\) of type \(\mathrm{G}_2\). Perhaps the most conceptual
description of \(\mathbf{G}\), due to Cartan in 1914, is as the
automorphism group of an octonion algebra: see \cite[Theorem
2.3.5]{SV:Octonions} for a modern algebraic proof.  A more suitable realization
of \(\mathbf{G}\) for our purposes, due to Engels in 1900, is as the stabilizer
of a generic \(3\)-form on a \(7\)-dimensional \(\kk\)-vector space \(V\); a
convenient choice of \(3\)-form, found in \cite[Eq. (3.13)]{Anderson:G2}, is
\[
\gamma \coloneqq
\mathrm{d}x_1 \wedge \mathrm{d}x_4 \wedge \mathrm{d}x_7 +
\mathrm{d}x_2 \wedge \mathrm{d}x_4 \wedge \mathrm{d}x_6 +
\mathrm{d}x_3 \wedge \mathrm{d}x_4 \wedge \mathrm{d}x_5 -
\mathrm{d}x_2 \wedge \mathrm{d}x_3 \wedge \mathrm{d}x_7 -
\mathrm{d}x_1 \wedge \mathrm{d}x_5 \wedge \mathrm{d}x_6.
\]

The simple roots \(\alpha_1\) and \(\alpha_2\) of \(\mathbf{G}\) have different
lengths. The homogeneous space for the maximal parabolic in \(\mathbf{G}\)
omitting the short root \(\alpha_1\) is isomorphic to a \(5\)-dimensional
quadric; omitting the long root \(\alpha_2\) instead gives the
\(\mathrm{G}_2\)-Grassmannian \(X\), which may be described as the closed
subvariety of \(\Gr(2,V)\) parameterizing \(2\)-dimensional spaces that are
isotropic for \(\gamma\):
\[
X \coloneqq
\{[U] \in \Gr(2,V) : \gamma(u,v,-) = 0\;\text{where \(U = \operatorname{span}\{u,v\}\)}\}.
\]
With the choice of \(\gamma\), the Schubert cell provided by the opposite
unipotent radical \(\mathbf{U}\) may be described, as is done in
\cite[Appendix D]{Anderson:Thesis}, in terms of the matrix
\[
M \coloneqq
\begin{pmatrix}
x_{\alpha_6} & -(x_{\alpha_3}x_{\alpha_5} + x_{\alpha_4}^2) & x_{\alpha_3} & x_{\alpha_4} & x_{\alpha_5} & 1 & 0 \\
x_{\alpha_2}x_{\alpha_4}- x_{\alpha_3}^2 & -(x_{\alpha_6} + x_{\alpha_3}x_{\alpha_4} + x_{\alpha_2}x_{\alpha_5}) & x_{\alpha_2} & x_{\alpha_3} & -x_{\alpha_4} & 0 & 1
\end{pmatrix}
\]
where the coordinates are indexed by their weights \(\alpha_2\),
\(\alpha_3 \coloneqq \alpha_1 + \alpha_2\),
\(\alpha_4 \coloneqq 2\alpha_1 + \alpha_2\),
\(\alpha_5 \coloneqq 3\alpha_1 + \alpha_2\), and
\(\alpha_6 \coloneqq 3\alpha_1 + 2\alpha_2\);
the weights may be deduced by comparing the relations amongst the
roots in type \(\mathrm{G}_2\) with the quadratic entries in \(M\).
A Levi factor is the group \(\mathbf{L} \cong \mathbf{GL}_2\) corresponding to
the root subsystem generated by \(\alpha_1\) and so, at least when \(p \geq 5\),
inspecting weights shows that the level decomposition of the adjoint
representation of \(\mathbf{L}\) on \(\mathfrak{u}\) decomposes as
\[
\mathfrak{u} \cong
\Sym^3\mathrm{Std}_{\mathbf{GL}_2} \otimes \kk_{-\alpha_2} \oplus
\kk_{-\alpha_6}.
\]

Globally, \(X\) may be realized as a set-theoretic linear section of the
Grassmannian, given by the non-transverse intersection
\(\Gr(2,V) \cap \PP\mathfrak{g}\), where the Lie algebra \(\mathfrak{g}\) of
\(\mathbf{G}\) canonically embeds into the ambient Pl\"ucker space via
the splitting \(\wedge^2 V \cong V \oplus \mathfrak{g}\) as
\(\mathbf{G}\)-representation: see \cite[\S22.3]{FH:Rep} for this latter
fact.  Better, the \(3\)-form \(\gamma\) induces a regular morphism of vector
bundles \(\wedge^2 \mathcal{S} \to \mathcal{Q}^\vee\) on \(\Gr(2,V)\), where
\(\mathcal{S}\) and \(\mathcal{Q}\) are the tautological sub- and quotient
bundles of ranks \(2\) and \(5\), respectively, whose vanishing locus
is precisely \(X\). The canonical bundle of \(X\) is therefore
\[
\omega_X
\cong \omega_{\Gr(2,7)} \otimes \det \mathcal{Q}^\vee(1)\rvert_X
\cong \sO_X(-3).
\]

\subsectiondash{}\label{homogeneous-G-computations}
To prove \parref{homogeneous-theorem} for \(X\), we must construct the elements
\(\partial_{\alpha_i}\tau\) as polynomials of degree \(2p-3\) in the minors of
\(M\).  We first construct the following element in the first irreducible
summand of \(\mathfrak{u} \cdot \tau\):
\[
-\partial_{\alpha_3}\tau
= x_{\alpha_3}^{p-2}(x_{\alpha_2} x_{\alpha_4} x_{\alpha_5} x_{\alpha_6})^{p-1}
\in \kk[\mathbf{U}_1] = \kk[x_{\alpha_2},\ldots,x_{\alpha_6}]/(x_{\alpha_2}^p,\ldots,x_{\alpha_6}^p).
\]
Since \(\alpha_3 = \alpha_1 + \alpha_2\) is indecomposable amongst the set of roots
\(\Phi^+_{\mathbf{P}} = \{\alpha_2,\alpha_3,\alpha_4,\alpha_5,\alpha_6\}\), the
displayed monomial spans the \(\lambda - \alpha_3 = 3(p-1)(3\alpha_1+2\alpha_2)-(\alpha_1+\alpha_2)\)
weight space in \(\kk[\mathbf{U}_1]\). This already implies that the following
weight \(\lambda - \alpha_3\) section of \(\sO_X(2p-3)\) is proportional to \(\partial_{\alpha_3}\tau\):
\[
(\det M_{1,7})^{p-3} 
(\det M_{1,5})
(\det M_{1,2})^{p-1} =
c  x_{\alpha_3}^{p-2}(x_{\alpha_2} x_{\alpha_4} x_{\alpha_5} x_{\alpha_6})^{p-1} \in \kk[\mathbf{U}_1].
\]
Thus the task is to show that the scalar \(c \in \kk\) is nonzero. Write the first \(p-2\) factors as
\[
(\det M_{1,7})^{p-3}(\det M_{1,5}) =
-x_{\alpha_4}x_{\alpha_6}^{p-2}
- x_{\alpha_5}x_{\alpha_6}^{p-3}(x_{\alpha_2}x_{\alpha_4} - x_{\alpha_3}^2)
\]
and separately compute the product of the two summands with
\[
(\det M_{1,2})^{p-1} =
\big(
-x_{\alpha_6}(x_{\alpha_6} + x_{\alpha_3}x_{\alpha_4} + x_{\alpha_2}x_{\alpha_5})
+ (x_{\alpha_3}x_{\alpha_5} + x_{\alpha_4}^2)(x_{\alpha_2}x_{\alpha_4} - x_{\alpha_3}^2)
\big)^{p-1}.
\]

In the product with \(-x_{\alpha_4}x_{\alpha_6}^{p-2}\), observe that only terms in
\((\det M_{1,2})^{p-1}\) which are linear in \(x_{\alpha_6}\) may contribute, and so
this gives the first line of
\begin{align*}
-x_{\alpha_4}x_{\alpha_6}^{p-2} (\det M_{1,2})^{p-1}
& =
-x_{\alpha_4}x_{\alpha_6}^{p-1}
(x_{\alpha_3}x_{\alpha_4} + x_{\alpha_2}x_{\alpha_5})
(x_{\alpha_3}x_{\alpha_5} + x_{\alpha_4}^2)^{p-2}(-x_{\alpha_3}^2 + x_{\alpha_2}x_{\alpha_4})^{p-2} \\
& = -x_{\alpha_3}^{p-2}(x_{\alpha_2}x_{\alpha_4}x_{\alpha_5}x_{\alpha_6})^{p-1}.
\end{align*}
Examining powers of \(x_{\alpha_2}\) shows that the nonzero contributions
from the three binomials must involve the \(x_{\alpha_2}x_{\alpha_5}\) in the
first binomial and the \((p-2)\)-nd power of \(x_{\alpha_2}x_{\alpha_4}\) in
the third binomial. With this, the second equality follows.

In the product with
\(-x_{\alpha_5}x_{\alpha_6}^{p-3}(x_{\alpha_2}x_{\alpha_4} - x_{\alpha_3}^2)\),
only terms in \((\det M_{1,2})^{p-1}\) which are quadratic in
\(x_{\alpha_6}\) may contribute, and so this gives
\begin{multline*}
-x_{\alpha_5}x_{\alpha_6}^{p-3}(x_{\alpha_2}x_{\alpha_4} - x_{\alpha_3}^2)
 \Big[x_{\alpha_6}^2(x_{\alpha_3}x_{\alpha_5} + x_{\alpha_4}^2)^{p-2}(x_{\alpha_2}x_{\alpha_4} - x_{\alpha_3}^2)^{p-2} \\
+ x_{\alpha_6}^2(x_{\alpha_3}x_{\alpha_4} + x_{\alpha_2}x_{\alpha_5})^2
  (x_{\alpha_3}x_{\alpha_5} + x_{\alpha_4}^2)^{p-3}(x_{\alpha_2}x_{\alpha_4} - x_{\alpha_3}^2)^{p-3}\Big].
\end{multline*}
For the first summand in the brackets, examining powers of \(x_{\alpha_2}\)
shows that it contributes a coefficient of \(-1\) to the
result. For the second summand, observe that
\((x_{\alpha_3}x_{\alpha_4} + x_{\alpha_2}x_{\alpha_5})^2\) must contribute
either a linear or quadratic term in \(x_{\alpha_5}\), and so its product
with the initial factor outside the brackets is
\[
-x_{\alpha_5}x_{\alpha_6}^{p-1}(x_{\alpha_2} x_{\alpha_4} - x_{\alpha_3}^2)^{p-2}(
2x_{\alpha_2}x_{\alpha_3}^{p-2}x_{\alpha_4}x_{\alpha_5}^{p-2}
-3x_{\alpha_2}^2x_{\alpha_3}^{p-4}x_{\alpha_4}^2x_{\alpha_5}^{p-2}) =
4x_{\alpha_3}^{p-2}(x_{\alpha_2}x_{\alpha_4}x_{\alpha_5}x_{\alpha_6})^{p-1}.
\]
In sum, this gives
\(
-x_{\alpha_5}x_{\alpha_6}^{p-3}(x_{\alpha_2}x_{\alpha_4} - x_{\alpha_3}^2)
(\det M_{1,2})^{p-1}
= 3x_{\alpha_3}^{p-2}(x_{\alpha_2}x_{\alpha_4}x_{\alpha_5}x_{\alpha_6})^{p-1}
\)
and so that \(c = -1 + 3 = 2\), which is nonzero since \(p \geq 5\).

It remains to construct the basis element
\(
-\partial_{\alpha_6}\tau =
(x_{\alpha_2}x_{\alpha_3}x_{\alpha_4}x_{\alpha_5})^{p-1} x_{\alpha_6}^{p-2} \in \kk[\mathbf{U}_1]
\).
Since \(\alpha_6 = \alpha_2 + \alpha_5 = \alpha_3 + \alpha_4\), the \(\lambda - \alpha_6\)
weight space of \(\kk[\mathbf{U}_1]\) is \(3\)-dimensional, and also contains the monomials
\(\partial_{\alpha_2} \partial_{\alpha_5}\tau\) and
\(\partial_{\alpha_3} \partial_{\alpha_4}\tau\). A straightforward computation
shows that
\begin{multline*}
(\det M_{1,7})^{p-2} (\det M_{1,2})^{p-1} =
(x_{\alpha_2}x_{\alpha_3}x_{\alpha_4}x_{\alpha_5})^{p-1} x_{\alpha_6}^{p-2}
\\+ (x_{\alpha_2}x_{\alpha_5}x_{\alpha_6})^{p-1}(x_{\alpha_3}x_{\alpha_4})^{p-2}
- 3(x_{\alpha_3}x_{\alpha_4}x_{\alpha_6})^{p-1}(x_{\alpha_2}x_{\alpha_5})^{p-2}.
\end{multline*}
To conclude, we show that \(\partial_{\alpha_2}\partial_{\alpha_5}\tau\) and
\(\partial_{\alpha_3}\partial_{\alpha_4}\tau\) also lie in the image of the
restriction map \(\res'\) using the action the Lie algebra \(\mathfrak{u}\).
Let \(X_{\alpha_6} = \partial_{\alpha_6}\) be the Lie algebra element
dual to \(x_{\alpha_6}\), and write
\[
X_{\alpha_2} = \partial_{\alpha_2} + c x_{\alpha_5}\partial_{\alpha_6}, \;\;\;
X_{\alpha_3} = \partial_{\alpha_3} + d x_{\alpha_4}\partial_{\alpha_6}, \;\;\;
X_{\alpha_4} = \partial_{\alpha_4} + d' x_{\alpha_3}\partial_{\alpha_6}, \;\;\;
X_{\alpha_5} = \partial_{\alpha_5} + c' x_{\alpha_2}\partial_{\alpha_6},
\]
for those corresponding to the root coordinates \(x_{\alpha_2}\),
\(x_{\alpha_3}\), \(x_{\alpha_4}\), and \(x_{\alpha_5}\), respectively.
Directly computing commutators of the derivations shows that
\[
[X_{\alpha_2}, X_{\alpha_5}] = (c' - c)\,X_{\alpha_6}
\;\;\text{and}\;\;
[X_{\alpha_3}, X_{\alpha_4}] = (d' - d)\,X_{\alpha_6}.
\]
Examining the structure constants in type \(\mathrm{G}_2\) shows that
\(c' - c\) is a nonzero multiple of \(\pm 1\), whereas \(d' - d\) is
a nonzero multiple of \(\pm 3\). Since the characteristic \(p\) is \(\geq 5\),
both are nonzero. Arguing as in step 2 of \parref{homogeneous-BD-submaximal}
now shows that \(\partial_{\alpha_2}\partial_{\alpha_5}\tau\) and
\(\partial_{\alpha_3}\partial_{\alpha_4}\tau\) lie in the image of the
restriction map. Combined with the determinant computation above, this
implies that \(\partial_{\alpha_6}\tau\) also lies in the image of
\(\res'\).  This completes the proof of \parref{homogeneous-theorem} in the
case of the \(\mathrm{G}_2\)-Grassmannian.
\qed

\section{Exceptions}\label{section-exceptions}
The exceptions in \parref{homogeneous-theorem} are truly exceptions: namely,
let \(X\) be either the symplectic Grassmannian \(\mathbf{SGr}(k,2m)\) for
\(2 \leq k \leq m\) with \(p = 2\); or the odd orthogonal Grassmannian
\(\mathbf{OGr}(m-1,2m+1)\) with \(p = 2\); or the \(\mathrm{G}_2\)-Grassmannian
with \(p = 2\) or \(p = 3\). In all of these cases, \(\mathcal{E}_X\)
is \emph{not} ample: see \parref{exceptions-symplectic-not-ample},
\parref{exceptions-orthogonal}, and \parref{exceptions-G2-not-ample},
respectively. In most cases, non-ampleness is explained by the existence of an
exotic isogeny between the relevant algebraic groups:
between \(\mathbf{SO}_{2m+1}\) and \(\mathbf{Sp}_{2m}\) in characteristic
\(p = 2\) for the first two cases, and for a group of type \(\mathrm{G}_2\)
with itself in characteristic \(p = 3\).

\subsectiondash{Orthogonal and symplectic Grassmannian}\label{exceptions-symplectic}
Suppose \(\kk\) has characteristic \(p = 2\), let \(V\) be a \(\kk\)-vector space
of dimension \(2m+1\), and let \(q \in \Sym^2 V^\vee\) be a nonsingular quadratic form,
by which we mean that the associated quadric hypersurface
\(Q \coloneqq \mathrm{V}(q) \subset \PP V\) is smooth. Since \(-1 = 1\) in \(\kk\),
the polar form \(b_q \colon V \times V \to \kk\) associated with \(q\), that
is, the symmetric bilinear form defined by
\[
b_q(u,v) \coloneqq q(u+v) - q(u) - q(v)
\;\;\text{for \(u,v \in V\),}
\]
is also alternating, and hence has even rank;
its \emph{radical}
\[
\operatorname{rad} b_q \coloneqq \{v \in V : b_q(u,v) = 0\;\text{for all \(u \in V\)}\}
\]
is therefore nonzero and nonsingularity of \(q\) implies that it is
\(1\)-dimensional, and \(b_q\) induces a nondegenerate symplectic form
\(\omega\) on the \(2m\)-dimensional quotient \(W \coloneqq V/\operatorname{rad} b_q\).
The representation of \(\mathbf{SO}(V,q)\) on \(W\) through the quotient map
\(\pi \colon V \to W\) preserves \(\omega\), yielding a purely inseparable
isogeny \(\mathbf{SO}(V,q) \to \mathbf{Sp}(W,\omega)\) of height \(1\) and
degree \(2^{2m}\): see \cite[\S23.5]{Borel} for more details. This induces
morphisms of the corresponding homogeneous spaces, the simplest being
the projection
\[
Q \coloneqq \{(x_0:\cdots:x_{2m}) \in \PP^{2m} : x_0x_1 + \cdots + x_{2m-2}x_{2m-1} + x_{2m}^2 = 0\}
\to \PP^{2m-1}
\]
from the strange point \((0:\cdots:0:1)\) of a smooth quadric in \(\PP^{2m}\)
to the \(\PP^{2m-1}\) of the first \(2m\) coordinates. More generally and
geometrically, we have:

\begin{Lemma}\label{exceptions-symplectic-projection}
The image under \(\pi \colon V \to W\) of a \(k\)-dimensional \(q\)-isotropic
subspace \(U \subset V\) is a \(k\)-dimensional \(\omega\)-isotropic subspace
\(\pi(U) \subset W\). The assignment \(U \mapsto \pi(U)\) defines a
\(\kk\)-morphism
\[
\varphi \colon \mathbf{OGr}(k,V,q) \to \mathbf{SGr}(k,W,\omega)
\]
which is purely inseparable of height \(1\), finite of degree \(2^k\), and
satisfies
\(\varphi^*\sO_{\mathbf{SGr}(k,W,\omega)}(1) \cong \sO_{\mathbf{OGr}(k,V,q)}(1)\).
\end{Lemma}

\begin{proof}
Nonsingularity of \(q\) implies that \(q\)-isotropic subspaces \(U \subset V\)
do not contain \(\operatorname{rad} b_q\). Combined with the fact that \(\omega\)
is induced by \(b_q\), it follows that \(\pi \colon V \to W\) sends \(q\)-isotropic
subspaces to \(\omega\)-isotropic ones of the same dimension. Applying \(\pi\)
to the universal subbundle on \(\mathbf{OGr}(k,V,q)\) induces
the morphism \(\varphi \colon \mathbf{OGr}(k,V,q) \to \mathbf{SGr}(k,W,\omega)\) of
schemes over \(\kk\); moreover, since \(\varphi\) is induced by linear projection, it
respects the Pl\"ucker line bundles.

It remains to show that \(\varphi\) is finite, purely inseparable of height \(1\),
and of degree \(2^k\). Consider a \(\kk\)-point of \(\mathbf{SGr}(k,W,\omega)\)
corresponding to a \(k\)-dimensional \(\omega\)-isotropic \(\bar{U} \subset W\),
and consider its scheme-theoretic inverse image \(\varphi^{-1}([\bar{U}])\).
Its \(\kk\)-points are given by \(q\)-isotropic subspaces in the
\((k+1)\)-dimensional space \(\pi^{-1}(\bar{U}) \subset V\). Since this
contains \(\operatorname{rad} b_q\) and since \(\bar{U}\) is isotropic for the
symplectic form \(\omega\) induced by \(b_q\), the restriction of \(q\) thereon
has rank \(1\), and so it contains a unique \(k\)-dimensional subspace \(U\)
such that \(\pi(U) = \bar{U}\). This already shows that \(\varphi\) is a
bijection on \(\kk\)-points, whence finite and purely inseparable: see
\citeSP{01S4}.

To understand the schematic structure of \(\varphi^{-1}([\bar{U}])\), it suffices
to consider its points valued in Artinian local \(\kk\)-algebras \(A\): these
are given by free rank \(k\) \(A\)-submodules of
\(\pi^{-1}(\bar{U}) \otimes_\kk A\) which are \(q\)-isotropic and which restrict to
the subspace \(U\) upon reduction modulo the maximal ideal. We may choose a basis
\(\pi^{-1}(\bar{U}) \otimes_\kk A \cong A u_1 \oplus \cdots \oplus A u_k \oplus A v\)
such that the restriction of \(q\) is the quadratic form \(y^2\), the square
of the coordinate dual to \(v\). With this, it is straightforward to see that
\[
\varphi^{-1}([\bar{U}])(A) \cong
\{A (u_1 + a_1 v) \oplus \cdots \oplus A (u_k + a_k v) : a_1,\ldots,a_k \in A\;\text{such that}\; a_i^2 = 0\}.
\]
This implies that
\(\varphi^{-1}([U]) \cong \operatorname{Spec} \kk[\alpha_1,\ldots,\alpha_k]/(\alpha_1^2,\ldots,\alpha_k^2)\),
from which the result follows.
\end{proof}

\begin{Proposition}\label{exceptions-symplectic-not-ample}
If \(p = 2\), then \(\mathcal{E}_X\) is not ample for any
\(X \coloneqq \mathbf{SGr}(k,2m)\) with \(2 \leq k \leq m\).
\end{Proposition}

\begin{proof}
Set \(Y \coloneqq \mathbf{OGr}(k,2m+1)\) and consider the morphism
\(\varphi \colon Y \to X\) from \parref{exceptions-symplectic-projection}. Its
Tschirnhausen bundle \(\mathcal{E}_\varphi\) is thus a vector bundle of rank
\(\deg \varphi - 1 = 2^k - 1\) and a Grothendieck--Riemann--Roch computation
using the canonical bundle formulae found in \parref{homogeneous-C} and
\parref{homogeneous-BD} shows that
\begin{align*}
c_1(\mathcal{E}_\varphi)
= -c_1(\varphi_*\sO_Y)
& = 2^{k-1} c_1(\mathcal{T}_X) - \frac{1}{2} \varphi_*c_1(\mathcal{T}_Y) \\
& = 2^{k-1}(2m+1-k) c_1(\sO_X(1)) - \frac{1}{2}(2m-k) \varphi_*c_1(\sO_Y(1))
\end{align*}
where \(\sO_X(1)\) and \(\sO_Y(1)\) are the Pl\"ucker line bundles on the two
isotropic Grassmannian. Since \(\varphi\) preserves Pl\"ucker line bundles, the
projection formula gives
\[
\varphi_*c_1(\sO_Y(1))
= \varphi_*c_1(\varphi^*\sO_X(1))
= \varphi_*\varphi^*c_1(\sO_X(1))
= 2^k c_1(\sO_X(1)).
\]
Therefore \(c_1(\mathcal{E}_\varphi) = 2^{k-1}c_1(\sO_X(1))\). Now, \(X\)
contains lines in its Pl\"ucker embedding, and the restriction of
\(\mathcal{E}_\varphi\) to any such is a vector bundle on \(\PP^1\) whose rank
\(2^k-1\) is smaller than its degree \(2^{k-1}\) whenever \(k \geq 2\): this is
not be ample, hence \(\mathcal{E}_\varphi\) is not ample, and thus
\(\mathcal{E}_X\) is not ample by \parref{setup-ample-insep-tschirnhausen}.
\end{proof}

\subsectiondash{}\label{exceptions-symplectic-dual}
Consider the morphism \(\psi \colon \mathbf{SGr}(k,W,\omega) \to \mathbf{OGr}(k,V,q)\)
dual to \(\varphi\); that is, the unique finite purely inseparable morphism of
height \(1\) such that \(\psi \circ \varphi\) and \(\varphi \circ \psi\) are
the Frobenius morphisms of \(\mathbf{OGr}(k,V,q)\) and
\(\mathbf{SGr}(k,W,\omega)\), respectively.  In particular, combined with the
properties of \(\varphi\) from \parref{exceptions-symplectic-projection}, this
implies that
\(\psi^*\sO_{\mathbf{OGr}(k,V,q)}(1) \cong \sO_{\mathbf{SGr}(k,W,\omega)}(2)\)
and that \(\deg \psi = 2^{N-k}\) where
\[
N
\coloneqq \dim \mathbf{OGr}(k,V,q) = \dim \mathbf{SGr}(k,W,\omega) =
\frac{1}{2}k(4m-3k+1).
\]
A Grothendieck--Riemann--Roch computation as in the proof of
\parref{exceptions-symplectic-not-ample} shows that
\[
c_1(\mathcal{E}_\psi)
= (2m-k-1)2^{N-k-2} c_1(\sO_{\mathbf{OGr}(k,V,q)}(1))
\]
In the case \(k = m - 1\), the restriction of \(\mathcal{E}_\psi\) to a line
has rank \(2^{N-m+1}-1\) and degree \(m 2^{N-m-1}\), whereby the argument of
\parref{exceptions-symplectic-not-ample} shows that the Frobenius-trace kernel
of \(\mathbf{OGr}(m-1,V,q) = \mathbf{OGr}(m-1,2m+1)\) is not ample when
\(m \leq 3\). In fact, \parref{exceptions-orthogonal} below shows that the
Frobenius-trace kernel is not ample for these odd orthogonal Grassmannian for
all \(m \geq 2\). The proof relies on finer geometric properties of the
morphism \(\psi\) which, in terms of foliations as in \cite{Ekedahl:Foliations}, says
that the leaves of the foliation associated with \(\psi\) are isomorphic
to the even orthogonal Grassmannian \(\mathbf{OGr}(k,2m)\):

\begin{Proposition}\label{exceptions-orthogonal-leaves}
If \(p = 2\), there exists, for any \(1 \leq k \leq m-1\), a Cartesian
commutative diagram
\[
\begin{tikzcd}
\mathbf{OGr}(k,2m) \rar \dar["F"'] & \mathbf{SGr}(k,2m) \dar["\psi"] \\
\mathbf{OGr}(k,2m) \rar & \mathbf{OGr}(k,2m+1)
\end{tikzcd}
\]
where the horizontal arrows are closed immersions. 
\end{Proposition}

\begin{proof}
Consider first the case \(k = 1\), so that \(\mathbf{SGr}(1,2m) = \PP^{2m-1}\)
and \(\mathbf{OGr}(1,2m+1) = Q\) is a smooth quadric in \(\PP^{2m}\). Choosing
coordinates as in \parref{exceptions-symplectic}, the morphism
\(\varphi \colon Q \to \PP^{2m-1}\) is linear projection onto the first
\(2m\) coordinates. The dual morphism \(\psi \colon \PP^{2m-1} \to Q\) is then
given by
\[
\psi(x_0:\cdots:x_{2m-1}) =
(x_0^2:\cdots : x_{2m-1}^2: x_0x_1 + \cdots + x_{2m-2}x_{2m-1}).
\]
Let \(H \subset \PP^{2m}\) be the hyperplane defined by \(x_{2m} = 0\) and
consider the hyperplane section \(Q' \coloneqq Q \cap H\). Linear projection
onto the first \(2m\) coordinates provides an isomorphism \(H \cong \PP^{2m-1}\),
whence a closed immersion \(Q' \subset \PP^{2m-1}\) identifying it as the quadric
defined by \(x_0x_1 + \cdots + x_{2m-2}x_{2m-1} = 0\).
Thus the restriction of \(\psi \colon \PP^{2m-1} \to Q\)
to \(Q' \subset \PP^{2m-1}\) factors through \(H \cap Q = Q'\) and is simply the
Frobenius morphism. The degree of both \(F \colon Q' \to Q'\) and
\(\psi \colon \PP^{2m-1} \to Q\) is \(2^{2m-2}\), so a standard argument
shows that the commutative diagram
\[
\begin{tikzcd}
Q' \rar \dar["F"'] & \PP^{2m-1} \dar["\psi"] \\
Q' \rar & Q
\end{tikzcd}
\]
is Cartesian. Recognizing \(Q' = \mathbf{OGr}(1,2m)\) now completes the proof
in the case \(k = 1\).

For \(k \geq 2\), we claim that the morphism dual to
\(\varphi \colon \mathbf{OGr}(k,2m+1) \to \mathbf{SGr}(k,2m)\)
from \parref{exceptions-symplectic-projection} is, in fact, the morphism
induced by \(\psi \colon \PP^{2m-1} \to Q\) given above. Observe first that
\(\psi\) sends a line \(\ell \subset \PP^{2m-1}\) isotropic for the symplectic
form
\(\omega = \mathrm{d}x_0 \wedge \mathrm{d}x_1 + \cdots + \mathrm{d}x_{2m-2} \wedge \mathrm{d}x_{2m-1}\)
to a line in \(Q\):  Given distinct points \(x = (x_0:\cdots:x_{2m-1})\) and
\(y = (y_0:\cdots:y_{2m-1})\) on \(\ell\) and any \(\lambda, \mu \in \kk\), the
polar identity applied to the quadratic form
\(q'(x) \coloneqq x_0x_1 + \cdots + x_{2m-2}x_{2m-1}\) defining \(Q'\) in
\(\PP^{2m-1}\) gives
\[
q'(\lambda x + \mu y) =
\lambda \mu b_{q'}(x,y) - 
\lambda^2 q'(x) - \mu^2 q'(y).
\]
Since the characteristic of \(\kk\) is \(p = 2\), the polar form \(b_{q'}\)
associated with \(q'\) is also symplectic, and is in fact none other than
\(\omega\). Since \(\ell\) is \(\omega\)-isotropic,
\(b_{q'}(x,y) = \omega(x,y) = 0\) and so
\[
\psi(\lambda x + \mu y)
= (\lambda^2x_0^2 + \mu^2 y_0^2 : \cdots : \lambda^2 x_{2m-1}^2 + y_{2m-1}^2: q'(\lambda x + \mu y))
= \lambda^2 \psi(x) + \mu^2\psi(y)
\]
meaning that \(\psi(\ell) \subset Q\) is a line. This implies that \(\psi\) sends
\(\omega\)-isotropic \((k-1)\)-planes in \(\PP^{2m-1}\) to \((k-1)\)-planes
in \(Q\), and so induces a morphism \(\mathbf{SGr}(k,2m) \to \mathbf{OGr}(k,2m+1)\);
abusing notation, this morphism will also be denoted by \(\psi\). A direct
calculation shows that \(\psi \circ \varphi\) and \(\varphi \circ \psi\) are the
endomorphisms of \(\mathbf{OGr}(k,2m+1)\) and \(\mathbf{SGr}(k,2m)\) which take
a \((k-1)\)-plane to itself, albeit squaring all the coordinates in the process:
these are the Frobenius maps, verifying that \(\psi\) is the dual of
\(\varphi\).

Finally, identify \(\mathbf{OGr}(k,2m)\) as the Fano scheme of linear spaces in
the quadric \((2m-2)\)-fold \(Q'\). Embed this into \(\mathbf{OGr}(k,2m+1)\)
and \(\mathbf{SGr}(k,2m)\) by viewing \(Q'\) as, respectively, a hyperplane
section of \(Q\) and as a quadric in \(\PP^{2m-1}\) whose polar form \(b_{q'}\)
is the symplectic form \(\omega\). The \(k = 1\) case shows that the restriction
to \(Q'\) of \(\psi \colon \PP^{2m-1} \to Q\) is the Frobenius morphism of \(Q'\),
so functoriality of Hilbert schemes implies that the restriction of
\(\psi \colon \mathbf{SGr}(k,2m) \to \mathbf{OGr}(k,2m+1)\) to the compatibly
embedded pair of \(\mathbf{OGr}(k,2m)\) is also the Frobenius, providing the
commutative diagram in the statement. That the diagram is also Cartesian now
follows from the fact that
\[
\deg(F \colon \mathbf{OGr}(k,2m) \to \mathbf{OGr}(k,2m)) =
2^{\frac{1}{2}k(4m-3k-1)} =
\deg(\psi \colon \mathbf{SGr}(k,2m) \to \mathbf{OGr}(k,2m+1)).
\qedhere
\]
\end{proof}

\begin{Proposition}\label{exceptions-orthogonal}
If \(p = 2\), then \(\mathcal{E}_Y\) is not ample for \(Y \coloneqq \mathbf{OGr}(m-1,2m+1)\).
\end{Proposition}

\begin{proof}
Let \(X \coloneqq \mathbf{SGr}(m-1,2m)\) and let \(\psi \colon X \to Y\) be
the morphism dual to that from \parref{exceptions-symplectic-not-ample}. By
\parref{exceptions-orthogonal-leaves}, there exists embeddings \(Z \subset X\)
and \(Z \subset Y\) of \(Z \coloneqq \mathbf{OGr}(m-1,2m)\) such that \(\psi\)
restricts to the Frobenius morphism \(F \colon Z \to Z\). Since \(\psi\) is
finite locally free, this implies that
\[
\psi_*\sO_X\rvert_Z \cong F_*\sO_Z
\]
yielding an isomorphism \(\mathcal{E}_\psi\rvert_Z \cong \mathcal{E}_Z\).
However, \(Z\) is a homogeneous space of Picard rank \(2\): In a nonsingular
quadratic space \((V',q')\) of dimension \(2m\), an isotropic subspace
\(U \subset V'\) of dimension \(m-1\) uniquely determines two \(m\)-dimensional
isotropic subspaces \(U_\pm \subset U^\perp \subset V'\) with the property that
\(U = U_+ \cap U_-\). The maps \([U] \mapsto [U_\pm]\) then exhibits \(Z =
\mathbf{OGr}(m-1,2m)\) as a projective bundle over the two spinor varieties
\(\mathbf{OGr}_\pm(m,2m)\) from \parref{homogeneous-spinor}. This implies that
\(\mathcal{E}_Z\) cannot be ample, for instance by
\parref{formal-finite-morphisms}, whence \(\mathcal{E}_\psi\) is not ample.
Therefore \(\mathcal{E}_Y\) is not ample by
\parref{setup-ample-insep-tschirnhausen}.
\end{proof}


\subsectiondash{Type \(\mathrm{G}_2\)}\label{exceptions-G2}
Consider now the case \(X\) is the \(\mathrm{G}_2\)-Grassmannian as in
\parref{homogeneous-G}. Since \(\omega_X^{-1} \cong \sO_X(3)\),
\parref{setup-numerical-obstruction} implies that \(\mathcal{E}_X\) is not
ample when \(p = 2\) since \(X\) contains lines for \(\sO_X(1)\): see
\cite[Lemma 3]{Strickland:Lines}, whose proof also works in characteristic
\(2\), though see also \cite{CC:Lines} and \cite[\S4.2]{LM:Lines}. The remaining case is more
interesting and pertains to the exotic isogeny
\(\phi \colon \mathbf{G} \to \mathbf{G}\) of algebraic groups of type
\(\mathrm{G}_2\) over fields of characteristic \(p = 3\).
Existence proofs using general methods from the theory of reductive groups
may be found, for example, in \cite[Chapter 10]{Steinberg} and \cite[Expos\'e XXI]{SGA3III}.
Chevalley's original construction is more explicit and better suited for our
purposes; we sketch it below and refer to
\cite[no.21]{Chevalley:Seminar}\footnote{See the original notes from
Chevalley's seminar at
{\scriptsize\url{https://www.numdam.org/item/SCC_1956-1958__2_/}}.} for
the details:

Let \(\mathbf{G}\) be a semisimple algebraic group of type \(\mathrm{G}_2\).
The basic observation is that, when \(\kk\) has characteristic \(p = 3\), the
adjoint representation \(\mathfrak{g}\) fits into a nonsplit short exact sequence
\[
0 \to V \to \mathfrak{g} \to V' \to 0
\]
where \(V\) is the \(7\)-dimensional representation of \(\mathbf{G}\) whose
highest weight is the first fundamental weight
\(\varpi_1 = \alpha_4 = 2\alpha_1 + \alpha_2\): this may be seen directly by
inspecting the multiplication table of the Lie algebra \(\mathfrak{g}\)---see
\cite[\S22]{FH:Rep} for instance---and observing that the root spaces of
weights \(\pm \alpha_1\), \(\pm \alpha_3\), and \(\pm \alpha_4\) generate a
\(p\)-Lie subalgebra due to vanishing of certain structure constants.
Chevalley shows this more conceptually by observing that the Killing form on
\(\mathfrak{g}\) degenerates when \(p = 3\) and identifies \(V\) as the radical
of the form. A useful consequence of this is that the induced bilinear form on
\(V'\) is nondegenerate and \(\mathbf{G}\)-equivariant, which may be used to
show that \(V'\) is also a simple module for \(\mathbf{G}\). This implies that
the morphism \(\mathbf{G} \to \mathbf{GL}(V')\) is an isogeny onto its image
\(\mathbf{G}'\). General results on semisimple algebraic groups now imply that
\(\mathbf{G}'\) is isomorphic to \(\mathbf{G}\), yielding the isogeny
\(\phi \colon \mathbf{G} \to \mathbf{G}' \cong \mathbf{G}\). Properties of the
induced morphisms on associated homogeneous spaces may be obtained from a
careful analysis of the isogeny \(\phi\); instead, we give a direct geometric
argument:

\begin{Lemma}\label{exceptions-G2-covers}
If \(p = 3\), there exists a dual pair of finite purely
inseparable morphisms of height \(1\)
\[
\varphi \colon X \to Q
\;\;\text{and}\;\;
\psi \colon Q \to X
\]
with
\(\deg\varphi = 9\) and \(\deg\psi = 27\), and such that
\(\varphi^*\sO_Q(1) \cong \sO_X(1)\) and \(\psi^*\sO_X(1) \cong \sO_Q(3)\).
\end{Lemma}

\begin{proof}
To construct the morphism \(\varphi \colon X \to Q\), observe that \(X\) naturally
embeds into \(\PP\mathfrak{g}\) and may be identified as the closed orbit of
\(\mathbf{G}\) acting on a vector of highest weight \(\alpha_6 = 3\alpha_1 + 2\alpha_2\).
That \(\mathfrak{g}\) contains the subrepresentation \(V\) means that
\(\PP\mathfrak{g}\) contains another \(\mathbf{G}\)-orbit whose closure is
\(\PP V\). Since \(X\) is not linearly degenerate in \(\PP\mathfrak{g}\), it
must be disjoint from the linear space \(\PP V\). Linear projection with
centre \(\PP V\) then induces a finite morphism \(X \to \PP V'\). This map is
\(\mathbf{G}\)-equivariant and so it factors through the closed
\(\mathbf{G}\)-orbit \(Q \subset \PP V'\), providing the morphism
\(\varphi \colon X \to Q\). Since \(\varphi\) is obtained from linear
projection, it satisfies \(\varphi^*\sO_Q(1) \cong \sO_X(1)\).

For the remaining properties, explicitly compute \(\varphi\) in terms of the
coordinates of the Schubert cell as in \parref{homogeneous-G}. For each pair
\(1 \leq i < j \leq 7\), let \(\wp_{i,j} \coloneqq \det M_{i,j}\) be the
Pl\"ucker coordinate associated with the \(i\)-th and \(j\)-th columns of the
coordinate matrix \(M\). By construction, \(\varphi\) projects onto the
Pl\"ucker coordinates of weights \(\pm \alpha_2\), \(\pm \alpha_5\),
\(\pm \alpha_6\), and one coordinate of weight \(0\); note that, as in
\parref{homogeneous-local-computations}, the weight of the Pl\"ucker
coordinates are related to that of the local affine coordinates
\(x_{\alpha_i}\) by a shift of \(-\alpha_6 = -3\alpha_1 - 2\alpha_2\).
A computation shows that
\[
\wp_{1,2}\wp_{6,7} + \wp_{1,3}\wp_{5,7} + \wp_{2,5}\wp_{3,6} + (\wp_{1,7} + \wp_{3,5})^2
= -3(x_{\alpha_2} x_{\alpha_5} x_{\alpha_6} + x_{\alpha_3} x_{\alpha_4} x_{\alpha_6}) = 0
\]
and so the map
\(\varphi \colon X \to Q \coloneqq \{(y_0:\cdots:y_6) \in \PP^6 : y_0y_6 + y_1y_5 + y_2y_4 +  y_3^2 = 0\}\)
is given by
\[
\varphi(x) =
(\wp_{1,2}(x):\wp_{1,3}(x): \wp_{2,5}(x): \wp_{1,7}(x) + \wp_{3,5}(x) : \wp_{3,6}(x) : \wp_{5,7}(x) : \wp_{6,7}(x)).
\]
In particular, \(\varphi\) restricts to a map of affine open subschemes
\(X \setminus \mathrm{V}(\wp_{6,7}) \to Q \setminus \mathrm{V}(y_6)\) corresponding
to the \(\kk\)-algebra map
\(\varphi^\# \colon \kk[y_1,\ldots,y_5] \to \kk[x_{\alpha_2},\ldots,x_{\alpha_6}]\) given by
\(y_4 \mapsto -x_{\alpha_2}\), \(y_5 \mapsto x_{\alpha_5}\),
\begin{gather*}
y_1 \mapsto x_{\alpha_3}^3 - x_{\alpha_2}x_{\alpha_3}x_{\alpha_4} + x_{\alpha_2}x_{\alpha_6},
\quad\quad
y_3 \mapsto x_{\alpha_6} - x_{\alpha_2}x_{\alpha_5} - x_{\alpha_3}x_{\alpha_4}, \\
y_2 \mapsto x_{\alpha_4}^3 - x_{\alpha_2}x_{\alpha_3}x_{\alpha_4} + x_{\alpha_2}x_{\alpha_5}^2 + x_{\alpha_5}x_{\alpha_6}.
\end{gather*}
Since \(y_1 + y_3y_4 - y_4^2y_5 \mapsto x_{\alpha_3}^3\) and
\(y_2 - y_3y_5 - y_4y_5^2 \mapsto x_{\alpha_4}^3\), it follows that \(\varphi\) is
purely inseparable of height \(1\) and degree \(9\). The properties
of \(\psi\) are then deduced from the relation \(\psi \circ \varphi = F\).
\end{proof}

\begin{Proposition}\label{exceptions-G2-not-ample}
If \(p = 2\) or \(p = 3\), then \(\mathcal{E}_X\) is not ample for
\(X\) the \(\mathrm{G}_2\)-Grassmannian.
\end{Proposition}

\begin{proof}
It remains to handle the characteristic \(3\) case. Let \(\psi \colon Y \to X\)
be the morphism from \parref{exceptions-G2-covers}. Then \(\mathcal{E}_\psi\) is a vector
bundle of rank \(\deg\psi - 1 = 26\) and first Chern class
\[
2c_1(\mathcal{E}_\psi)
= -2c_1(\psi_*\sO_Y)
= 27 c_1(\mathcal{T}_X) - \psi_*c_1(\mathcal{T}_Y)
= 81 c_1(\sO_X(1)) - 5\psi_*c_1(\sO_Y(1)).
\]
Since \(\psi^*\sO_Y(1) \cong \sO_X(3)\), a projection formula computation
shows that \(\psi_*c_1(\sO_Y(1)) = 9c_1(\sO_X(1))\), from which it follows that
\(c_1(\mathcal{E}_\psi) = 18c_1(\sO_X(1))\). As before, \(X\) contains lines
in its Pl\"ucker embedding, the restriction of \(\mathcal{E}_\psi\) onto
any such line is a vector bundle of rank \(r = 26\) and degree \(d = 18\),
and so cannot be ample. Therefore \(\mathcal{E}_\psi\) is itself
not ample, whence \(\mathcal{E}_X\) is not ample by \parref{setup-ample-insep-tschirnhausen}.
\end{proof}

It would be interesting to describe \(\varphi \colon X \to Q\) appearing in
\parref{exceptions-G2-covers} in terms of the trilinear form \(\gamma\) used to
define \(X\) in \parref{homogeneous-G}. Towards this, it may be useful to view
\(\mathbf{G}\) as the automorphism group of the split octonionic Cayley algebra
\(C = \kk \oplus V\) over \(\kk\), where the \(7\)-dimensional representation
\(V\) may be identified as the subspace of purely imaginary octonions. When \(p
= 3\), the commutator with respect to the octonionic product satisfies the
Jacobi identity---see \cite[p.402]{BEMNM} and also \cite[Remark
3.4.1]{Manivel:BRIDGES}---making \(C\) into a Lie algebra.  Identifying
\(\mathfrak{g}\) as the Lie algebra of derivations on \(C\), the adjoint
representation of \(C\) on itself then provides an embedding of Lie algebras
\(V \subset C \to \mathfrak{g}\).

\bibliographystyle{amsalpha}
\bibliography{main}
    \end{document}